\documentclass[11pt,leqno]{article}
\usepackage[nullcolor]{mymacros}
\usepackage{cite}

\newcommand{\Cbeta}{{\sf M}(\beta)}

\title{Kawasaki dynamics beyond the uniqueness threshold}
\author{Roland Bauerschmidt\footnote{Courant Institute of Mathematical Sciences, NYU. E-mail: {\tt bauerschmidt@cims.nyu.edu}.}
\and Thierry Bodineau\footnote{IHES, CNRS. E-mail: {\tt bodineau@ihes.fr}.}
\and Benoit Dagallier\footnote{Courant Institute of Mathematical Sciences, NYU. E-mail: {\tt bd2543@cims.nyu.edu}.}}

\date{\vspace*{-2em}} 

\begin{document}
\maketitle
\begin{abstract}
  Glauber dynamics of the Ising model on a random regular graph is known to mix fast below the
  tree uniqueness threshold and exponentially slowly above it.
  We show that Kawasaki dynamics of the canonical ferromagnetic
  Ising model on a random $d$-regular graph mixes fast beyond the tree uniqueness threshold when $d$ is large enough
  (and conjecture that it mixes fast up to the tree reconstruction threshold for all $d\geq 3$).
  This result follows from a more general spectral condition for (modified) log-Sobolev inequalities
  for conservative dynamics of Ising models.
  The proof of this condition in fact extends to perturbations of distributions with log-concave generating polynomial.
\end{abstract}

\section{Introduction and results}

We first discuss our results for Kawasaki dynamics on random regular graphs, which are our primary motivation,
in Section~\ref{sec:intro-rrg},
and then the more general results for conservative dynamics of Ising models that these follow from,
in Section~\ref{sec:intro-general}.
{\red In Section~\ref{sec:logconcave}, we state a generalisation (with the same proof)
  to interacting measures on the bases of a matroid which could be of independent interest.}

\subsection{Result for random regular graphs}
\label{sec:intro-rrg}

Let $G=([N],E)$ be a random $d$-regular graph on $N$ vertices.
Glauber dynamics of the ferromagnetic Ising model on $G$ is known to mix fast when $\beta<\beta_c$ and exponentially slowly when $\beta>\beta_c$,
where $\beta_c =\operatorname{artanh}(1/(d-1)) \sim 1/(d-1)$ is the uniqueness threshold for the Ising
model on the infinite $d$-regular tree,
see \cite{MR3059200} for fast mixing (or \cite[Example~6.15]{2307.07619} for a perspective closer to that of this work)
and \cite{4389492,MR2650042} for slow mixing.
Different from the situation on finite-dimensional lattices, 
the Ising model on the infinite $d$-regular tree
has a second phase transition  which occurs at the reconstruction threshold
$\beta_r = \operatorname{artanh}(1/\sqrt{d-1}) \sim 1/\sqrt{d-1}$, see for example \cite{MR1391195} and the further discussion below.
While the uniqueness transition of the ferromagnetic Ising model is related to the magnetisation (the average of all spins),
which is an order parameter for this transition on random regular graphs,
the magnetisation becomes irrelevant for the canonical Ising model, which is the Ising model conditioned on its magnetisation {\red (see~\eqref{e:Ising} below for a definition)}.
The natural analogue of Glauber dynamics for the canonical Ising model is Kawasaki dynamics. 
{\red This is a Markov chain reversible with respect to the canonical Ising measure, in discrete or continuous time},  
which 
randomly swaps neighbouring spins and thus conserves the total magnetisation.

Our following main result shows that, for $d \geq d_0$,  Kawasaki dynamics of the {\red ferromagnetic} 
Ising model on a random $d$-regular graph mixes fast  beyond the tree uniqueness threshold,
at which Glauber dynamics of the slows down exponentially.
{\red A more detailed version of the following theorem is given in Corollary~\ref{cor:rrg}.}

\begin{theorem} \label{thm:rrg}
  For $d\geq 3$ and $\beta< 1/(8\sqrt{d-1})$,
  Kawasaki dynamics of the canonical {\red nearest-neighbour} Ising model
  on a random $d$-regular graph on $N$ vertices
    (quenched, and ferromagnetic or antiferromagnetic)
  mixes in $O_{d,\beta}(N\log^6\!N)$ steps,
  with high probability on the randomness of the graph.
\end{theorem}
{\red Throughout the paper Kawasaki dynamics will be defined only through its Dirichlet form. Since in our setting all jump rates are bounded, this implies that all our results apply to both continuous or discrete time versions of the dynamics up to rescaling time.} 

{\red Theorem~\ref{thm:rrg} implies} in particular, since $\beta_c = \operatorname{artanh}(1/(d-1))<1/(8\sqrt{d-1})$ for $d\geq d_0$ large, 
that indeed Kawasaki dynamics is fast beyond $\beta_c$.
Our result also applies under a more general condition which we expect (but do not prove) holds for $\beta<\beta_r$, see the later discussion.
Further we expect that the power of the logarithm is $1$ instead of $6$ in the optimal mixing time estimate.
As part of the proof, we obtain this mixing
time for a dynamics in which all exchanges of spins are permitted rather than only those of neighbours. 
{\red In addition, our results in fact extend to Ising models defined on more constrained configuration spaces (bases of a matroid). 
We focus on the canonical Ising model in the paper and indicate the required changes subsequently.} 

{\red
\paragraph{Interpretation of Theorem~\ref{thm:rrg}}

Before commenting the  relaxation of the Kawasaki dynamics on  the $d$-regular graph, we first recall the more studied case of the Kawasaki dynamics for the Ising model with ferromagnetic interactions  on a lattice in $\bbZ^d$. 
On a $d$-dimensional cube of size $L$ and for temperatures such that the Gibbs measure has strong mixing properties, the continuous time Kawasaki dynamics reaches equilibrium on a natural diffusive time scale $L^2$ \cite{MR1233852,MR1414837,MR1757965,MR1914934},  while the Glauber dynamics is much  faster (the time scale to reach equilibrium is $1$) \cite{MR1746301}.
These two time scales are due to different  mechanisms: for the Glauber dynamics the relaxation to equilibrium is achieved locally in a more or less independent way across the lattice, while for the Kawasaki dynamics, the spins (viewed as particles and holes) move around as weakly interacting random walks and transporting mass on a distance $L$ requires a diffusive time of order $L^2$.
For a temperature below the (unique) critical temperature of the Ising model, a phase transition occurs for the Gibbs measure and the  relaxation time becomes much longer for the Glauber dynamics \cite{MR1746301} as well as  for the Kawasaki dynamics;
this is known at least in $d=2$ \cite{MR1705586,Beltran_landim}.

As stated in Theorem~\ref{thm:rrg}, the Kawasaki dynamics behaves very differently on random regular graphs.
The fast relaxation at high temperature is intuitive as the  diffusive behaviour of the spin motion on the lattice is replaced by the much faster ballistic motion of a simple random walk on a random regular graph.
More surprisingly, the Kawasaki dynamics on random regular graphs can be much \emph{faster} than the Glauber dynamics at least for large degree, i.e., when $\beta_c = \operatorname{artanh}(1/(d-1))<1/(8\sqrt{d-1})$ by Theorem~\ref{thm:rrg},
and conjecturally for all $\beta_c<\beta_r = \operatorname{artanh}(1/\sqrt{d-1})$. This
shows in some sense that, on random regular graphs, the magnetisation is the \emph{only} bottleneck responsible
for the critical slowdown of the Glauber dynamics at the uniqueness threshold $\beta_c$.
Since the magnetisation is fixed in the Kawasaki dynamics, it is not sensitive to this bottleneck.
A similar phenomenon takes place for the Ising model on the complete graph, where at any temperature Kawasaki dynamics is simply the symmetric simple exclusion process.

Nonetheless, the hard constraint on the magnetisation is well known to make Kawasaki dynamics significantly more difficult to study
than Glauber dynamics. Moreover, the established approaches to Kawasaki dynamics on lattices are based on comparison
with the Ising model without constraint through mixing conditions \cite{MR1233852,MR1414837,MR1757965,MR1914934},
and one therefore cannot hope that these methods extend beyond the uniqueness threshold where the canonical and unconstrained models
behave differently.
The derivation of Theorem~\ref{thm:rrg} relies instead on more global spectral estimates. 
Using this point of view,
the two thresholds $\beta_c\sim 1/d$ and $\beta_r\sim 1/\sqrt{d}$ are parallel to the two largest eigenvalues
$\lambda_1=d$ and $\lambda_2 \sim 2\sqrt{d}$ of the adjacency matrix of the random regular graph and the existence
of a second threshold is the analogue of random regular graphs being expanders.
Our perspective also provides a different approach to Kawasaki dynamics on lattices.}

\paragraph{About the condition on $\beta$}

The conclusion of Theorem~\ref{thm:rrg} applies either under the stated assumption $\beta<1/(8\sqrt{d-1})$ 
{\red which is a spectral condition on the adjacency matrix of random regular graphs, see below~\eqref{condition_SC}}, 
or under a more general condition on the covariance matrices of the canonical Ising model stated in (CC) below.
We expect but do not prove that the latter condition holds for all $\beta$ up to the reconstruction threshold $\beta_r$
which we also expect to be a unique critical point for the canonical Ising model on the random regular graph,
but are not aware of results that show this, see also the discussion below.
A proof that Kawasaki dynamics remains fast up to $\beta_r$
(either by establishing the  covariance condition or in another way) would be very interesting.

Under essentially the same assumption as ours, namely $\beta<1/(C\sqrt{d-1})$, and with a somewhat related motivation,
it was recently shown that the Ising model on the infinite $d$-regular tree is a factor of IID \cite{MR4369724}.

\paragraph{About the canonical Ising model}

Ising models on random regular graphs have been studied extensively, mostly due to their  motivation in the spin glass context \cite{MR2518205}.
This motivation is even more relevant for the canonical Ising model as the constraint on the magnetisation can lead to frustration
while the randomness of the graph produces disorder,
and indeed, it is expected that the canonical Ising model displays spin-glass like behaviour at low temperatures.
In support of this, it has been proved that in the limit $d\to\infty$ and $\beta\to\infty$ the quenched free energy
of the canonical Ising model (with magnetisation $0$) coincides
with that of the Sherrington--Kirkpatrik model given by the Parisi formula \cite{MR3630296}.
More generally, for finite $\beta$ we expect that similar considerations show
that the quenched free energy $F_\beta$ of the canonical Ising model
with {\red magnetisation $m=0$, i.e., spin configurations satisfy $\sum_i \sigma_i =0$, see \eqref{e:Ising} below,}
satisfies, as $d\to\infty$ with $\beta/\sqrt{d-1}$ of order $1$ in our normalisation (discussed below),
\begin{equation}
  F_{\beta/\sqrt{d-1}} = {\sf P}_\beta + O_{\beta/\sqrt{d-1}}(\frac{1}{\sqrt{d}}),
\end{equation}
where ${\sf P}_\beta$ is the free energy of the Sherrington--Kirkpatrik model
given by the Parisi solution. In particular, ${\sf P}_\beta$ is given by the
replica symmetric solution for $\beta<1$ and nontrivial for $\beta>1$. Thus it
seems natural to expect that the canonical Ising model has a related unique transition
at $\beta/\sqrt{d-1} = \beta_r/\sqrt{d-1} \sim 1$, but we are not aware of results
proving this. It would be especially interesting to construct
a test function that shows that Kawasaki dynamics indeed becomes slow beyond this transition.
The vanishing of the spectral gap of Glauber dynamics of
{\red several mean-field spin glass models} at low temperature was proved in \cite{MR3825934}.

The above picture for the canonical Ising model on random regular graphs
is different from that of the standard Ising model (without constraint of the magnetisation),
whose behaviour we briefly summarise now.
As a consequence of the local convergence of the random $d$-regular graph to the infinite $d$-regular tree,
it has been shown at any temperature that the Ising model on a random regular graph converges weakly to the symmetric mixture of the plus and minus states
of the Ising model on the infinite $d$-regular tree
and further that the Ising model conditioned on positive magnetisation converges to the plus state on the infinite $d$-regular tree \cite{MR2875752}.
However, while the local correlations of the Ising model on the random regular graph are closely related to those of the infinite tree,
its global behaviour is more subtle.
It was shown \cite{4389492} that the reconstruction problem for the Ising model on the random $d$-regular graph is solvable for $\beta > \beta_c = \operatorname{artanh}(1/(d-1))$, the tree uniqueness threshold,
whereas the reconstruction threshold on the infinite $d$-regular tree is $\beta_r=\operatorname{artanh}(1/\sqrt{d-1})$.
The solution of the reconstruction problem for the unconditioned Ising model on random $d$-regular graph is mediated through the
magnetisation and its solvability amounts to showing that the latter is a bottleneck for Glauber dynamics.
Therefore, for the unconditioned Ising model on a random regular graph, there is a standard ferromagnetic phase transition
at the tree uniqueness threshold $\beta_c$ (and this transition is well understood).

\paragraph{Further related literature}

Glauber dynamics of Ising models on finite and infinite trees has been studied in detail \cite{MR2123248,MR2094519},
but we emphasise again that the global behaviour on finite trees is quite different from that of random regular graphs.
{\red The behaviour of Ising models on finite trees is dominated by the boundary (which does not look like a regular tree).}
Results on Glauber dynamics of random regular graphs include \cite{MR3059200,4389492} which establish fast and slow mixing
below respectively above the tree uniqueness threshold $\beta_c$ and also \cite{MR4282194}
which investigates metastability properties. Swenden--Wang and block dynamics are also known to mix fast up to the tree uniqueness threshold
\cite{MR4505379}.

For the spin-glass Ising model on a random regular graph, where each edge has an independent $\pm 1$ coupling constant
(but without the constraint on the magnetisation),
it is known that Glauber dynamics mixes fast under the condition $\beta<1/(4\sqrt{d-1})$, and again fast mixing is expected up to
$\beta_r$ which is, in fact, also the critical point for the spin glass Ising model on the infinite $d$-regular tree \cite{MR853978}.
The fast mixing of Glauber dynamics for $\beta < 1/(4\sqrt{d-1})$ was observed in \cite{MR4408509} as a consequence of the
recent spectral criteria for the Glauber dynamics log-Sobolev
inequality \cite{MR3926125}, spectral gap \cite{MR4408509}, or the modified log-Sobolev inequality
\cite{2106.04105}, and is analogous to the simple criterion $\beta<1/4$ for the Sherrington--Kirkpatrik model from \cite{MR3926125}.
Recent advances in sampling the Sherrington--Kirkpatrik model up to $\beta<1$ (but not using Glauber dynamics) include \cite{2203.05093}. 
{\red Finally, there has been recent work on efficient sampling algorithms for canonical Ising models on general graphs of bounded maximal degree~\cite{Carlson2021ComputationalTF} (not using Kawasaki dynamics).
In particular, at zero magnetisation, it is shown that no algorithm applies to all bounded degree graphs beyond the tree uniqueness threshold.
This is not in contradiction with our results for \emph{random} graphs (or more generally suitably strong expander graphs), which shows  that
Kawasaki dynamics remains fast beyond the tree uniquess threshold.
It is also shown in \cite{Carlson2021ComputationalTF} that for magnetisations bounded away from $0$ (depending on the temperature) efficient algorithms exist for all temperatures.}

\subsection{General results for conservative dynamics of Ising models}
\label{sec:intro-general}

\subsubsection{Canonical Ising model and conservative dynamics}

Let $A$ be a symmetric matrix with constant eigenvector ${\bf 1} = (1,\dots, 1)$, i.e.,
$A {\bf 1} = \lambda_1 {\bf 1}$, and remaining eigenvalues $\lambda_2 \leq \cdots \leq \lambda_N$ ({\red not necessarily bigger than $\lambda_1$}\footnote{We
  will sometimes refer to $\lambda_1$ as the smallest eigenvalue and to $\lambda_2$ as the second smallest one,
because this is the situation for Laplacian matrices of regular graphs, but in principle the value of $\lambda_1$
has no significance for the following.}).
We write $\delta(A) = \lambda_N-\lambda_2$ for the length of the support of the nontrivial spectrum.
The canonical Ising measure with coupling matrix $A=(A_{ij})$ and external field $h=(h_i)$ is defined by
\begin{equation} \label{e:Ising}
  \E_{\nu_{\beta,h}}[F] \propto \sum_{\sigma\in\Omega_{N,m}} e^{-\frac{\beta}{2}(\sigma,A\sigma) + (h,\sigma)} F(\sigma),
\end{equation}
where, for $m \in [-1,1]$ such that $Nm$ is an integer,
\begin{equation}
  \Omega_{N,m} = \{  \sigma \in \{\pm 1\}^N: \sum_{i=1}^N \sigma_i=Nm \},
\end{equation}
{\red and $\propto$ denotes proportionality up to a normalisation constant (independent of $F$).
In the following we will often write $\nu$ as a shorthand for $\nu_{\beta,h}$.}

{\red We define now conservative dynamics on $\Omega_{N,m}$ such that  configurations $\sigma \in \Omega_{N,m}$ are updated to new configurations denoted by $\sigma^{ij} \in\Omega_{N,m}$ for $i,j\in[N]$ and  obtained by exchanging spins  $\sigma_i$ and $\sigma_j$. 
For later purposes, we  consider conservative dynamics reversible with respect to general measures $\mu$ on $\Omega_{N,m}$ with  jump rates denoted by $(c(\sigma,\sigma'))_{\sigma,\sigma'} = (c_\mu(\sigma,\sigma'))_{\sigma,\sigma'}$.
We will always assume that $c(\sigma,\sigma')=0$ unless $\sigma'=\sigma^{ij}$ for some $i,j \in [N]$
(not necessarily neighbours). The Dirichlet form associated with these rates is defined as 
\begin{equation} 
\label{e:Dirichlet- general}
D_\mu^{c} (F,G) = 
\frac{1}{2}\sum_{\sigma'} \E_\mu \qa{c_\mu(\sigma,\sigma') (F(\sigma')-F(\sigma))
\p{G(\sigma')-G(\sigma)}},
\end{equation}
and we always set $D_\mu^c(F)= D_\mu^c(F,F)$.}

{\red The speed of relaxation of these dynamics can be quantified by functional inequalities which are defined below.}
Given  any measure $\mu$ on $\Omega$ and a function $F : \Omega \to \R_+$, the relative entropy and variance are defined by
\begin{align}
  \label{e:ent-def}
  \ent_\mu(F) &= \E_{\mu}[\Phi(F)] -\Phi(\E_\mu[F]), \qquad \Phi(x)=x\log x,\\
  \label{e:var-def}
  \var_\mu(F) &= \E_{\mu}[\Phi(F)] -\Phi(\E_\mu[F]), \qquad \Phi(x)=x^2.
\end{align}
{\red Given also a Dirichlet form $D_\mu^c$, we use the convention that the log-Sobolev constant $\gamma$,
the modified log-Sobolev constant $\gamma_m$, and the spectral gap $\lambda$ are given
as the best constants in the inequalities
\begin{align}
  \ent_\mu(F) &\leq \frac{2}{\gamma} D_\mu^c(\sqrt{F})\\
  \ent_\mu(F) &\leq \frac{1}{2\gamma_m} D_\mu^c(F,\log F)\\
  \var_\mu(F) &\leq \frac{1}{\lambda} D_\mu^c(F)
\end{align}
so that always $\gamma \leq \gamma_m \leq \lambda$, see for example \cite{MR2283379}.}

\subsubsection{Spectral conditions}

In all of our results, we assume either of the following conditions. {\red Note that neither condition explicitly depends on the number $N$ of vertices}.

\medskip\noindent
\textbf{Spectral condition (SC).}
Let $\delta(A) = \lambda_N-\lambda_2$ be the length of the support of the nontrivial spectrum of the coupling matrix $A$. 
{\red We say that the spectral condition (SC) holds} if $\beta < 1/(2\delta(A))$, and then set
\begin{equation}
  \Cbeta
  = \frac{1}{1 -2\beta\delta(A)}.\label{condition_SC}
\end{equation}

{\red In the special case of the canonical nearest-neighbour Ising model on a random $d$-regular graph, 
the spectral condition (SC) is satisfied for $\beta<1/(8\sqrt{d-1})$ and the constant $\Cbeta$ is of order $1$ in $N$, see the proof of Corollary~\ref{cor:mp-rrg}.}

\medskip\noindent
\textbf{Covariance condition (CC).}
  Let $\bar\chi^0(\beta)$ be an upper bound on the largest eigenvalue of the covariance matrix $(\cov_{\nu_{\beta,h}}(\sigma_i;\sigma_j))_{i,j}$
  of the canonical Ising model on $\Omega_{N,m}$, uniformly in $h\in\R^N$, i.e.,
  \begin{equation}
    \red
    \bar\chi^0(\beta) = \sup_{h \in \R^N} \sup_{|f|_2=1} \var_{\nu_{\beta,h}}((\sigma,f)),
  \end{equation}
  and set
  \begin{equation} \label{condition_CC}
    \Cbeta
    = \exp\qa{\int_0^\beta \bar\chi^0(t)\, dt}
    .
\end{equation}

The spectral condition (SC) is easy to verify and applies at sufficiently high temperatures.
The more general covariance condition (CC) 
{\red trivially holds with an $N$-dependent constant since $\bar\chi^0(t)\leq N$ for any $t$ as spins are bounded by $1$ in absolute value. 
The point is to get an estimate on $\bar\chi^0(t)$ that has better dependence on the number $N$ of sites. 
In particular, in the proof of Theorem~\ref{thm:mlsi}, we will show that the covariance condition (CC)}
is implied by (SC) {\red with $\Cbeta$ of order $1$ in $N$ when $\beta<1/(2\delta(A))$}.

\begin{remark}
\label{rem: Glauber}
The Glauber dynamics analogue of the spectral condition appeared in \cite{MR3926125} for the log-Sobolev inequality
and other functional inequalities under this condition were later considered in \cite{MR4408509,2106.04105,2203.04163}.
Compared to the previously existing high temperature conditions for functional inequalities,
the spectral condition on the coupling matrix for Glauber dynamics covers the SK model up to $\beta<1/4$
(which is still the best known condition for fast relaxation of Glauber dynamics for the SK model, but see \cite{2203.05093} for
a different sampling strategy up to $\beta<1$).
The Glauber version of the spectral condition involves the \emph{smallest} eigenvalue $\lambda_1$ instead of the \emph{second smallest} eigenvalue $\lambda_2$ here.

The Glauber dynamics analogues {\red of the covariance condition} appeared in \cite{MR4705299} and \cite{2203.04163}, see also the review \cite{2307.07619}.
{\red For ferromagnetic spin couplings, i.e., $A_{ij}\leq 0$ for $i\neq j$,
the  uniformity in the external field condition is automatically as a consequence of \cite{MR4586225}.}
For Glauber dynamics of {\red ferromagnetic} Ising models {\red the covariance condition (CC) applies in a very general setting up to the critical point with
constant $\Cbeta$ uniform in $N$}.  
It moreover implies a polynomial dependence of the log-Sobolev constant near the critical temperature
under the mean-field bound on the susceptibility which holds on $\Lambda \subset \Z^d$ if $d>5$, see \cite{MR4705299}.
\end{remark}

{\red 
\subsubsection{Main results for mean-field dynamics}

A canonical choice of jump rates are those of the standard (or mean-field) Dirichlet form
which are $c_{\mu}(\sigma,\sigma') = \frac{1}{2N}(1+ \mu (\sigma')/ \mu (\sigma))$ and the corresponding   Dirichlet form is simply denoted by 
\begin{equation}
  D_\mu (F) 
  = \frac{1}{2}\sum_{\sigma'} \E_\mu \qa{c_\mu (\sigma,\sigma') (F(\sigma)-F(\sigma'))^2}
  = \frac{1}{2N} \sum_{\sigma'} \E_\mu \qa{(F(\sigma)-F(\sigma'))^2}.
\end{equation}
The notation $D_{\mu}$ for a measure $\mu$ on $\Omega_{N,m}$ henceforth always denotes the standard Dirichlet form.
Our results also apply to the Dirichlet form of the down-up walk, see Sections~\ref{sec:inftemp}--\ref{sec:thm_mlsi}.}

\begin{theorem}\label{thm:mlsi}
  {\red Let $N\geq 1$ and $m\in[-1,1]$ be such that $Nm$ is an integer.} 
  Assume either (SC) or (CC) {\red and let $\Cbeta$ be the corresponding constant.}
  Assume further 
  $\max_i\sum_{j\neq i} |A_{ij}| \leq \bar A$ and $\max_i |h_i| \leq \bar h$ for some $\bar A,\bar h>0$.  
Then, 
{\red for any test function $F:\Omega_{N,m}\to\R_+$, }
  the canonical Ising model {\red $\nu = \nu_{\beta,h}$} satisfies
  \begin{equation} \label{e:mlsi}
	  \ent_\nu(F) \leq \Cbeta C(\beta \bar A,\bar h) D_\nu(F, \log F),
  \end{equation}
where 
  \begin{equation} \label{e:Dirichlet-standard}
    D_\nu(F,G) = \frac{1}{2N}\sum_{i,j=1}^N\E_{\nu}\qB{\p{F(\sigma)-F(\sigma^{ij})}\p{G(\sigma)-G(\sigma^{ij})}}.
  \end{equation}
  {\red The same result  with constant $C(\beta)$ independent of $\bar A,\bar h$ instead of $C(\beta\bar A,\bar h)$ holds if the standard Dirichlet form \eqref{e:Dirichlet-standard} is replaced by that of the down-up walk with jump rates \eqref{e:downup}.}
\end{theorem}

{\red None of the constants in Theorem~\ref{thm:mlsi} explicitly depend on $N$. In particular, if $\beta\bar A,\bar h$ and $\Cbeta$ are bounded uniformly in $N$ then so are the modified log-Sobolev constants.

It was recently shown in a general setting that the modified
log-Sobolev inequality (with Dirichlet form $D_\nu(F,\log F)$) implies the usual log-Sobolev inequality (with Dirichlet form  $D_\nu(\sqrt{F}) = D_\nu(\sqrt{F}, \sqrt{F})$) at the cost of a factor that is $O_{\beta\bar A,\bar h}(\log N)$ in our situation \cite{MR4620718}.
In Corollary \ref{cor:lsi} below, we deduce from Theorem \ref{thm:mlsi} a spectral gap estimate of order 1 in $N$ and a 
log-Sobolev inequality with an additional logarithmic factor, by using \cite{MR4620718}.
We stress that the factor $\log N$ for the (unmodified) log-Sobolev constant
is necessary for the validity of a uniform estimate with respect to the density $m \in [-1,1]$. Indeed for the standard Dirichlet form \eqref{e:Dirichlet-standard}, the log-Sobolev inequality
is not true with constant of order $1$ in $N$ without restrictions on $m$ and $h$,
see \cite[Theorem 5]{MR1675008} which shows a logarithmic correction for $m$ very close to $\pm 1$.}

\begin{corollary} \label{cor:lsi}
  Under the same assumption as in Theorem~\ref{thm:mlsi} (and with the same Dirichlet form), writing $D_\nu(F)$ for $D_\nu(F,F)$:
      \begin{equation} \label{e:sg-cor}
    \var_\nu(F) \leq \Cbeta C(\beta \bar A,\bar h) D_\nu(F),
  \end{equation}
  and
  \begin{equation} \label{e:lsi-cor}
    \ent_\nu(F) \leq \Cbeta C(\beta \bar A,\bar h) (\log N) D_\nu(\sqrt{F}).
  \end{equation}
\end{corollary}

The constants $C(\beta\bar A,\bar h)$ in the different inequalities can be different, but $\Cbeta$ is always
the constant in   either condition (SC) or (CC).

\begin{proof}
  It is a standard fact that the modified log-Sobolev inequality implies the spectral gap inequality \eqref{e:sg-cor}.
  {\red The main result of \cite{MR4620718} then implies the log-Sobolev inequality \eqref{e:lsi-cor} with constant $20 C_{\rm mLSI}\log (1/c_{\min})$, 
  where $C_{\rm mLSI}$ is the constant in the right-hand side of~\eqref{e:mlsi} and $c_{\min}$ is the minimal transition probability associated with the dynamics of the standard Dirichlet. 
  This $c_{\min}$ satisfies $1/c_{\min} = O_{\beta \bar A,\bar h}(N)$ under our assumptions. }
\end{proof}

\subsubsection{Main results for Kawasaki dynamics}

The above inequalities apply to the standard Dirichlet form \eqref{e:Dirichlet-standard} or the down-up Dirichlet form \eqref{e:downup}, both
associated with dynamics in which exchanges for all pairs of spins are permitted. 
Kawasaki dynamics will refer to nearest-neighbour jump rates reversible with respect to $\nu_{\beta,h}$,
provided that $[N]$ is identified with the vertices of a graph so that the notion of nearest-neighbours makes sense.
{\red The corresponding Dirichlet form is
\begin{equation}
  \label{eq: Dirichlet Kawasaki}
D^{\mathrm{K}}_\nu(F) = \sum_{i\sim j}\E_\nu\qa{(\sqrt{F}(\sigma)-\sqrt{F}(\sigma^{ij}))^2},
  \end{equation}
where the exchange are restricted to neighbouring sites on the graph, denoted by $i\sim j$.}

For the usual log-Sobolev inequality and the spectral gap, the change from the standard Dirichlet form
to the nearest-neighbour Dirichlet form
can be achieved by comparison inequalities, called moving particle lemmas. 
{\red Theorem~\ref{thm:rrg} as well as estimates for Kawasaki dynamics on the $d$-dimensional lattice} 
follow from Corollary~\ref{cor:lsi} and versions of the moving particle lemma
(combined with simple estimates of the geometry in the case of the random regular graph).
{\red This is the content of the next two corollaries, with Corollary~\ref{cor:rrg} a more precise statement of Theorem~\ref{thm:rrg}.}

\begin{corollary}
\label{cor:rrg}
Let $d\geq 3$. For the canonical nearest-neighbour Ising model {\red  (ferromagnetic or antiferromagnetic)} on a random $d$-regular graph on $N$ vertices with $\beta < 1/(8 \sqrt{d-1})$,
  the inverse log-Sobolev constant {\red $1/\gamma$} of Kawasaki dynamics  is bounded by $C(d,\beta,h)\log^5\! N$ {\red with high probability on the randomness of the graph as $N\to\infty$}.
  Hence Kawasaki dynamics mixes in at most $O_{d,\beta,h}(N\log^6\! N)$ steps {\red with high probability}.
\end{corollary}

\begin{corollary} \label{cor:zd}
  Let $d\geq 1$, let $\Lambda \subset \Z^d$ be a hypercube of side length $L$,
  and assume that either condition (SC) or (CC) holds.
  {\red Assume that there are $\bar A,\bar h>0$ such that $\max_i \sum_{j \neq i}|A_{ij}|\leq \bar A$ and $\max_i |h_i|\leq \bar h$}, and that
  the finite-range condition $A_{ij} = 0$ if $\dist(i,j)>R$ for some $R<\infty$ holds.

  Then, uniformly in the magnetisation $m$, the inverse spectral gap {\red $1/\lambda$} of Kawasaki dynamics on $\Lambda$ is bounded by $O_{\beta\bar A,\bar h}(L^2)$ 
  and the inverse log-Sobolev constant {\red $1/\gamma$} bounded by $O_{\beta\bar A,\bar h}(L^2\log L)$. 
\end{corollary}

\begin{proof}[\red Proof of Theorem~\ref{thm:rrg} \& Corollary~\ref{cor:rrg}]
  For the ferromagnetic (respectively antiferromagnetic) Ising model on a random regular graph on $[N]$,
  the coupling matrix $A$ is minus the adjacency matrix (respectively the adjacency matrix) of the random regular graph.
  Thus $-A$ has Perron--Frobenius eigenvalue $d$ with eigenvector ${\bf 1}$, and it is known that
  the remaining eigenvalues are contained in $[-2\sqrt{d-1}-\epsilon,2\sqrt{d-1}+\epsilon]$ for any $\epsilon>0$,
  with probability tending to $1$ as $N\to\infty$, see \cite{MR2437174}.
The spectral condition (SC) therefore holds for $\beta < 1/(8\sqrt{d-1})$, with high probability.
  Thus Theorem~\ref{thm:mlsi} implies the uniform modified log-Sobolev inequality \eqref{e:mlsi} with respect to the mean-field Dirichlet form,
  and by Corollary~\ref{cor:lsi},
  the (unmodified) log-Sobolev inequality (still with respect to the mean-field Dirichlet form) holds with an additional factor $O_{\beta,h}(\log N)$:
  \begin{equation}
    \ent_\nu(F) \leq C(d,\beta\bar A,\bar h) \frac{\log N}{N} \sum_{i,j}\E_\nu\qa{(\sqrt{F}(\sigma)-\sqrt{F}(\sigma^{ij}))^2}.
  \end{equation}
  The comparison inequality from Corollary~\ref{cor:mp-rrg} now allows to replace the mean-field Dirichlet form by the Kawasaki  {\red Dirichlet form \eqref{eq: Dirichlet Kawasaki}:}
  \begin{equation}
    \ent_\nu(F) \leq C(d,\beta\bar A,\bar h) \log^5\!N 
    D^{\mathrm{K}}_\nu(F).
  \end{equation}
  
  Finally, it is well known that the mixing time is controlled by the log-Sobolev constant $\gamma$ \cite{MR1490046}:
  \begin{equation}
    t_{\text{mix}}(1/e)
    \leq \frac{2}{\gamma} \pa{1+\frac14 \log\log\pa{\max_\sigma\frac{1}{\nu(\sigma)}}}
  \leq \frac{2}{\gamma} \pa{1+\frac14 \log N + C\big(\beta\delta(A),h\big)},
  \end{equation}
  where we used that, for $\sigma \in \Omega_{N,m}$,
  \begin{equation}
    \log \frac{1}{\nu(\sigma)} \leq N\log 2 + \big(\beta \delta(A) + 2\max_i|h_i|\big)N.
  \end{equation}
  This completes the proof.
\end{proof}

\begin{proof}[Proof of Corollary~\ref{cor:zd}]
  By Theorem~\ref{thm:mlsi} and Corollary~\ref{cor:lsi}, we again have
  \begin{align}
    \var_\nu(F) &\leq \Cbeta C(\beta\bar A,\bar h) \frac{1}{L^d}\sum_{i,j}\E_\nu\qa{(\sqrt{F}(\sigma)-\sqrt{F}(\sigma^{ij}))^2},
    \\
    \ent_\nu(F) &\leq \Cbeta C(\beta\bar A,\bar h) \frac{\log L}{L^d} \sum_{i,j}\E_\nu\qa{(\sqrt{F}(\sigma)-\sqrt{F}(\sigma^{ij}))^2},
  \end{align}
  and the proof is completed by the moving particle lemma, Corollary~\ref{cor:mp-Zd}.
\end{proof}

{\red
\subsection{Extension to interacting measures on the bases of a matroid}
\label{sec:logconcave}

The proof of Theorem~\ref{thm:mlsi} also implies the following generalisation.
The set of \emph{bases of a matroid} on $[N]$ is any nonempty collection  $\cB$ of subsets of $[N]$ 
satisfying the basis exchange property:
if $A,B \in \cB$ and $a\in A\setminus B$ then there is $b\in B\setminus A$ such that $(A\setminus \{a\})\cup\{b\} \in \cB$.
In particular, all $A \in \cB$ have the same number of elements.
Examples of bases of a matroid are $\Omega_{N,m}$ where $\sigma \in \Omega_{N,m}$ is identified with the set of $+$ spins $I(\sigma) = \{i \in [N]: \sigma_i =1\}$
(the bases of the uniform matroid)
or the set of spanning trees of a finite graph (the bases of the graphic matroid).
See Section~\ref{sec:inftemp} for the definition of the down-up walk
(also known as bases-exchange walk)
on the bases of a matroid and the associated Dirichlet form denoted by $D^{\mathrm{du}}_\mu$.

\begin{theorem} \label{thm:matroid}
Let $\pi$ be the uniform distribution on the bases of a matroid, viewed as a probability measure on $\omega \in \{0,1\}^N$,
and consider the perturbed probability measure
\begin{equation}
  \mu(\omega) \propto \prod_{e,f}(1+\epsilon_{ef} \omega_e\omega_f) \pi(\omega),
\end{equation}
where $\epsilon_{ef} > -1$ for $e,f\in [N]$.
There is $\bar\epsilon_0>0$ independent of $\pi$ such that if $\bar\epsilon = \max_e \sum_{f} |\epsilon_{ef}|\leq \bar\epsilon_0$ then
the down-up walk on $\mu$ satisfies a uniform modified log-Sobolev inequality:
\begin{equation} 
  \ent_\mu(F) \leq C(\bar\epsilon) D^{\mathrm{du}}_\mu(F,\log F).
\end{equation}
Moreover, a spectral condition involving the matrix
$(\log(1+\epsilon_{ef}))_{e,f\in [N]}$
can be formulated as well.
\end{theorem}

Tyler Helmuth pointed out that the following application of the above extension is of interest:
let $\pi$ be the uniform spanning tree measure on a graph so that $\omega_e=1$ if $e$ is an edge in the spanning tree and $0$ otherwise.
The above example gives a modified log-Sobolev inequality for weakly non-uniform spanning trees, e.g., 
measures on spanning trees where parallel edges are favoured.
}

\subsection{Proof overview}

The starting point for our results is the strategy of the proof of the Glauber log-Sobolev inequality under a spectral condition from \cite{MR3926125}
and much further developed in \cite{MR4303014,MR4705299,MR4408509,2106.04105,2203.04163}.
The set-up is recalled and adapted to the setting of conservative dynamics in Section~\ref{sec:main}.
For a general introduction, also see \cite{2307.07619}.

The constraint on the magnetisation makes the analysis much more subtle, though, and several new ingredients are required.
An important input is a (modified or unmodified) log-Sobolev inequality
for the infinite temperature case $\beta=0$, for which
we make use of the results for the down-up walk (or bases exchange walk) from \cite{MR4203344,2106.04105} which are based
on the deep log-concavity results of \cite{MR2476782,1811.01600,MR4172622}.
These are also recalled in Section~\ref{sec:main},
and a different presentation of the proof of the modified log-Sobolev inequality for the down-up
walk is included in Appendix~\ref{app:downup}.

A crucial new ingredient is a comparison estimate that connects the down-up walk and Kawasaki dynamics
through the measure decomposition of Section~\ref{sec:main}.
This estimate is given in Section~\ref{sec:downup}.

In Section~\ref{sec:mp}, we recall the moving particle lemma and apply it to the case of the random regular graph
using well-known estimates on the geometry of these graphs.

{\red All arguments are carried out for the canonical Ising model but we indicate along the way the changes required for the more general setting of Theorem~\ref{thm:matroid}.}

\subsection{Open questions}

Our result leaves several natural questions open which we summarise here.

1. Can one extend the fast mixing result for Kawasaki dynamics on random regular graphs from 
$\beta < 1/( 8 \sqrt{d-1})$ to $\beta<\beta_r$, for example by establishing the
covariance condition (CC)?

2. Can one show that Kawasaki dynamics on random regular graphs becomes slow for $\beta>\beta_r \red = \operatorname{artanh}(1/\sqrt{d-1}) \sim 1/\sqrt{d-1}$ (or at least for $\beta$ sufficiently large)
by choice of trial function in
the spectral gap inequality? See \cite{MR1705586} for such results on finite-dimensional lattices and
\cite{MR3825934,2203.05093} for hard sampling results for the related SK model.

3. Can one remove the logarithmic factor in the log-Sobolev inequality \eqref{e:lsi-cor} for suitable choice of jump rates depending on $h$ and $m$
which are equivalent to those of the standard Dirichlet form when $h$ is bounded and $m \in [-1+\epsilon,1-\epsilon]$?

4. Can one remove the logarithmic factor in the Dirichlet form comparison estimate for random regular graphs, Corollary~\ref{cor:mp-rrg}?
This would imply that Kawasaki dynamics has spectral gap of order~$1$ on random regular graphs (rather than with a logarithmic correction),
and therefore make this estimate suitable to deduce correlation decay \cite{MR1971582}.

{\red 5. For the Kawasaki dynamics on a $d$-dimensional cube of size $L$,
previous methods \cite{MR1233852,MR1414837,MR1757965,MR1914934} relying on 
mixing conditions provide a decay of order $L^{-2}$ for the spectral gap and the log-Sobolev constant.
Our approach, which uses spectral conditions instead, is less precise in the lattice case for the log-Sobolev constant
and it would be interesting to remove the logarithmic corrections in Corollary~\ref{cor:zd}.}

6. On a lattice $\Lambda\subset \Z^d$, 
can one show that the covariance condition (CC) gives a polynomial gap near $\beta_c$ when $d\geq 5$,
and in general that it applies up to $\beta < \beta_c$? Moreover, can one show that (CC) follows from spatial mixing assumptions?

\section{Proof \red of main theorems} 
\label{sec:main}

We {\red focus the presentation on the proof of Theorem~\ref{thm:mlsi} concerning the canonical Ising model.
The immediate generalisations for interacting measures on general bases of matroids leading to Theorem~\ref{thm:matroid} are
discussed subsequently in Section~\ref{sec:pfmatroid}.}

\subsection{Initial decomposition of the measure}\label{sec_def_nu_r_pi}

For $m \in [-1,1]$ such that $Nm$ is an integer, recall that
\begin{equation}
  \Omega_{N,m} = \{  \sigma \in \{\pm 1\}^N: \sum_{i=1}^N \sigma_i=Nm \},
\end{equation}
and also define the continuous analogue
\begin{equation}
  X_{N,0} = \{\varphi \in \R^N: \sum_{i=1}^N\varphi_i=0\}.
\end{equation}
The proof of Theorem~\ref{thm:mlsi} starts from a decomposition of the canonical Ising measure $\nu$ on $\Omega_{N,m}$ into two measures: 
an infinite temperature part $\pi$ with random external field on $\Omega_{N,m}$
and a renormalised measure $\nu_0$ on $X_{N,0}$ 
which determines the statistics of the external field.
This decomposition is analogous to the decompositions that we used in \cite{MR3926125,MR4705299}
for Glauber dynamics of Ising models,
which is also equivalent to the one implicit in \cite{2203.04163},
with the constraint on the magnetisation the only but crucial difference.
The subsequent additional difficulties and improvements (in the high temperature condition to use the
second smallest instead of the smallest eigenvalue) are due to this constraint.

To define this decomposition,
{\red recall that the matrix $A$ is assumed to have the constant eigenvector ${\bf 1}= (1,\dots, 1)$ and that it therefore preserves its orthogonal complement $X_{N,0}$.
Moreover, we}
may from now on assume that $A$ regarded as an operator $X_{N,0} \to X_{N,0}$ has spectrum contained in $(0,1)$.
Indeed, by replacing $A$ by $A+cP$ with $c > -\lambda_2$, where $P$ is the orthogonal projection onto $X_{N,0}$,
we may assume without loss of generality that $A$ is positive definite on $X_{N,0}$
(the shift by $cP$ does not affect the canonical Ising measure since $\sigma_i^2=1$ for all $i$),
and by rescaling $\beta$ we can further assume $|A|_{X_{N,0}}<1$ where $|A|_{X_{N,0}}$ is the operator norm of $A$ on $X_{N,0}$.
By taking a limit, all estimates extend to the case in which $A$ has spectrum in $[0,1]$ rather than $(0,1)$.

\newcommand{\alphab}{\beta}

{\red Generalising the decomposition of the canonical Ising measure into its infinite temperature version and a renormalised measure,
it is useful to consider its decomposition into a canonical Ising measure at inverse temperature $t \in [0,\beta)$
and a corresponding renormalised measure $\nu_t$.
To define this decomposition, for $t \in [0,\beta]$, set}
\begin{equation}
\label{eq: covariance decomposition}
  C_t = (tA + (\alphab-t)P)^{-1}
\end{equation}
where the inverse is taken on $X_{N,0}$ 
and note that $C_t$ is strictly increasing as a quadratic form in $t\in[0,\beta]$ {\red since $|A|_{X_{N,0}}<1$.}
The following Gaussian convolution identity then holds:
for any $\sigma\in \Omega_{N,m}$
and $t\in [0,\beta)$,
\begin{equation} \label{e:decomp-t}
  e^{-\frac{\beta}{2}(\sigma,A\sigma)}
  \propto
  e^{-\frac{1}{2}(\sigma,C_\beta^{-1}\sigma)}
  \propto
  \int_{X_{N,0}} e^{-\frac{1}{2} (\sigma-\varphi,C_t^{-1}(\sigma-\varphi))} e^{-\frac12 (\varphi, (C_\beta-C_t)^{-1} \varphi)}\, d\varphi,
\end{equation}
where all inverses are as operators on $X_{N,0}$ and are extended to act trivially orthogonal to it,
and $\propto$ denotes proportionality up to a constant independent of $\sigma$ but dependent on $t$, $\beta$, {\red and $m$.
Indeed, one can for example notice that all matrices can be diagonalised simultaneously, and that the integral 
then becomes a product of $N-1$ one-dimensional integrals.}
For $t=0$, in particular, 
\begin{equation} \label{e:decomp-0}
  e^{-\frac{\beta}{2}(\sigma,A\sigma)}
  \propto
  \int_{X_{N,0}} e^{-\frac{\alphab}{2} (\sigma-\varphi,\sigma-\varphi)} e^{-\frac12 (\varphi, C^{-1} \varphi)}\, d\varphi
\end{equation}
where $C=C_\beta-C_0 = (\beta A)^{-1}-\beta^{-1}P$. The canonical Ising measure then decomposes as
\begin{equation}
  \E_{\nu}[F] = \E_{\nu_0}[\E_{\pi_{N,m,h+\alphab\varphi}}[F]].
  \label{eq_decomp_Ising}
\end{equation}
Here $\pi_{N,m,h}$ denotes the infinite temperature canonical Ising model with external field $h \in \R^N$, i.e.,
the probability measure on $\Omega_{N,m}$ with probabilities
\begin{equation}  \label{e:pih-def}
  \pi_{N,m,h}(\sigma) \propto \prod_{i=1}^N e^{\sigma_i h_i}.
\end{equation}
The renormalised measure $\nu_0$ on $X_{N,0}$ has density proportional to
\begin{equation} \label{e:nu0-def}
\nu_0 (\varphi) \propto  e^{-\frac12(\varphi,C^{-1}\varphi)-V_0(\varphi)},
\end{equation}
where the renormalised potential $V_0$  is given by 
\begin{equation}
  V_0(\varphi) = - \log \sum_{\sigma\in\Omega_{N,m}} e^{-\frac{\alphab}{2}(\varphi-\sigma,\varphi-\sigma)} \pi_{N,m,h}(\sigma).
\end{equation}
Using \eqref{e:decomp-t} instead of \eqref{e:decomp-0}, the decomposition   \eqref{eq_decomp_Ising} generalises to
\begin{equation}
  \E_{\nu}[F] = \E_{\nu_t}[\E_{\mu_{t}^\varphi}[F]],
  \label{eq_decomp_Ising_t}
\end{equation}
where the renormalised measure $\nu_t$ on $X_{N,0}$ and the fluctuation measure $\mu_{t}^\varphi$ on $\Omega_{N,m}$, for $\varphi \in X_{N,0}$,
are defined analogously by
\begin{alignat}{2}
  \label{e:nut-def}
  \frac{d\nu_t}{d\varphi}(\varphi) &\propto e^{-\frac12 (\varphi, (C_\beta-C_t)^{-1}\varphi) - V_t(\varphi)}, &\qquad& (\varphi \in X_{N,0}),
  \\
  \label{e:mut-def}
  \mu_{t}^\varphi(\sigma) &  \propto e^{-\frac12 (\sigma-\varphi, C_t^{-1} (\sigma-\varphi))}\pi_{N,m,h}(\sigma)
                            \propto e^{-\frac12 (\sigma, C_t^{-1} \sigma)}\pi_{N,m,h+C_t^{-1}\varphi}(\sigma), &\qquad& (\sigma \in \Omega_{N,m}),
\end{alignat}
with $V_t(\varphi) = -\log \sum_{\sigma\in\Omega_{N,m}} e^{-\frac12 (\varphi-\sigma, C_t^{-1}(\varphi-\sigma))}\pi_{N,m,h}(\sigma)$ the renormalised potential.
Thus $\mu_0^\varphi= \pi_{N,m,\beta \varphi+h}$ and $\mu_t^\varphi$ is again a canonical Ising measure at inverse temperature $t$
with general field $C_t^{-1}\varphi +h$.

{\red
\begin{remark} \label{rk:decomp-matroid}
  The above holds without change if the canonical Ising measure is replaced by its generalisation
  in which $\Omega_{N,m}$ is replaced by a subset of $\Omega_{N,m}$.
  In particular, $\pi_{N,m,h}$ is then a probability measure on this subset.
\end{remark}
}

\subsection{Infinite temperature measure}
\label{sec:inftemp}

We first collect properties of the infinite temperature measure $\pi_{N,m,h}$.

\subsubsection{Covariance estimates}

{\red It is well known that the canonical measure $\pi_{N,m,h}$ (independent Bernoulli random variables conditioned on their sum)} satisfies the strong Rayleigh property \cite{MR2476782}.
In particular, it is negatively correlated for all $h\in \R^\Lambda$:
\begin{equation} \label{e:nc}
  \cov_{\pi_{N,m,h}}(\sigma_i,\sigma_j) \leq 0 \qquad (i \neq j),
\end{equation}
and hence its covariance matrix satisfies (using also $2|f_if_j| \leq f_i^2+f_j^2$ and $\sum_{j: j\neq i} \sigma_j = -\sigma_i +Nm$):
\begin{align}\label{e:var-bd}
  \var_{\pi_{N,m,h}}((f,\sigma))
  &= \sum_{i} f_i^2 \var_{\pi_{N,m,h}}(\sigma_i) + \sum_{i \neq j} f_if_j \cov_{\pi_{N,m,h}}(\sigma_i,\sigma_j)
  \nnb
  &\leq \sum_{i} f_i^2 \var_{\pi_{N,m,h}}(\sigma_i) + \sum_{i} f_i^2 \sum_{j: j \neq i} \cov_{\pi_{N,m,h}}(\sigma_i,-\sigma_j)
  \nnb
  &= 2 \sum_{i} f_i^2 \var_{\pi_{N,m,h}}(\sigma_i)
  \leq 2 |f|_2^2, 
  \quad \text{with $|f|_2^2 = \sum_i f_i^2$}.
\end{align}
A related (but different) consequence  of the strong Rayleigh property  {\red other than the negative correlation \eqref{e:nc}} is
the log-concavity of the generating polynomial
\begin{equation}
  \red z \in [0,\infty)^N \mapsto g_{\red \pi_{N,m,h}} (z) = \E_{\red \pi_{N,m,h}}  [z^I], \qquad \text{where } z^I = \prod_{i\in I}z_i,
  \label{eq_def_gpi}
\end{equation}
see  \cite{MR4332671,MR4172622}.
Note that the log-concavity property  is also established for much more general measures.
The log-concavity of $g_\pi(z)$ has been used to prove a modified log-Sobolev inequality for the so-called
\emph{down-up} or \emph{bases-exchange walk} \cite{2106.04105,MR4203344} which we define next.

\subsubsection{Down-up walk}

Another choice of jump rates we will use are those of
the \emph{down-up walk} (also called \emph{bases-exchange walk}) studied in \cite{MR4232133,2106.04105,MR4203344}.
For $\sigma\in\Omega_{N,m}$, let $I(\sigma)=\big\{i \in[N]:\sigma_i=1\big\}$ denote the set
of sites with $+$ spins (particles).
The down-up walk acts by removing a particle uniformly at random (down step) to obtain a distribution
with $k-1$ particles and then selects a $k$ particle distribution according to $\pi$ from those containing
the former distribution.
{\red Denote by $J_i(\sigma)$ the set of admissible sites where a particle can be added after removing one at site $i$.
  In particular, $J_i(\sigma)=I(\sigma)^c \cup \{i\}$ for the product measure conditioned on its sum.
}
In terms of spins,
{\red the jump rates of the down-up walk correspond to}:
\begin{equation} 
\label{e:downup}
{\red c^{\mathrm{du}}} (\sigma,i\to j) 
=
{\red c^{\mathrm{du}}_{\pi_{N,m,h}}} (\sigma,i\to j) 
=
{\bf 1}_{\sigma_i=1}{\bf 1}_{\sigma_j=-1}\frac{\pi(\sigma^{ij})}{\sum_{k\in \red J_i(\sigma)} \pi(\sigma^{ik})}.
\end{equation}
{\red The Dirichlet form \eqref{e:Dirichlet- general} corresponding to these rates will be denoted by $D^{\mathrm{du}}_\pi$.}

We emphasise that the rates of the down-up walk are asymmetric in $(i,j)$, i.e., ${\red c^{\mathrm{du}}} (\sigma,i\to j) \neq {\red c^{\mathrm{du}}} (\sigma,j\to i)$
  -- in fact if the left-hand side is nonzero the right-hand side is zero. 
  This should simply be understood as the fact that, if a particle ($+$ spin) and a hole ($-$ spin) exchange position, then we think of the particle as making the jump rather than the hole.
  Nonetheless the jump rates are reversible with respect to $\pi$
since $\pi(\sigma) {\red c^{\mathrm{du}}} (\sigma,i\to j) = \pi(\sigma^{ji}) {\red c^{\mathrm{du}}} (\sigma^{ji},j\to i)$ where on the right-hand side we wrote $\sigma^{ji}=\sigma^{ij}$
to emphasise that the plus particle of $\sigma^{ij}$ (among the sites $i,j$) is at $j$ if the one of $\sigma$ is at $i$. 
In the following the value of the jump rate will always be clear from the context, so we will use the more convenient notation:
\begin{equation}
{\red c^{\mathrm{du}}} (\sigma,\sigma^{ij}) 
:= 
{\red c^{\mathrm{du}}} (\sigma,i\to j)
.
\end{equation}

More precisely, {\red for the canonical Ising model,}
we use the down-up walk if $m\leq 0$ (i.e., the number of particles is at most $N/2$),
and instead the {\red up-down walk} for $m>0$ whose jump rates would be ${\red c^{\mathrm{ud}}}(\sigma,i\to j) = {\red c^{\mathrm{du}}_{\pi_{N,m,-h}}}(-\sigma,i\to j)$.
Note that the number $N-k+1$ of terms in the sum in the denominator of the jump rate is always of order $N$
(this is the reason for considering different jump rates depending on the sign of $m$).
In particular, if $C^{-1} \pi(\sigma) \leq \pi(\sigma^{ij})\leq C\pi(\sigma)$ for all $i,j\in [N]$,
then
\begin{equation} 
\label{e:equiv-downup-standard}
  {\red c^{\mathrm{du}}}(\sigma,\sigma^{ij}) \leq C^2  \frac{{\bf 1}_{\sigma_i=1}{\bf 1}_{\sigma_j=-1}}{N-k+1}
  \leq \frac{2C^2}{N} {\bf 1}_{\sigma_i=1}{\bf 1}_{\sigma_j=-1}.
\end{equation}
Thus under the condition $C^{-1}\pi(\sigma)\leq \pi(\sigma^{ij}) \leq C \pi(\sigma)$ the
Dirichlet forms of the down-up walk {\red $D^{\mathrm{du}}_\pi$} and the standard Dirichlet form 
{\red $D_\pi$} are equivalent,
but the jump rates become inequivalent for large fields.

The following modified log-Sobolev inequality for $\pi_{N,m,h}$
has been proven as a consequence of the log-concavity of its generating polynomial \cite{MR4203344,2106.04105}.
We have included an alternative presentation of the proof in Appendix~\ref{app:downup}.

\begin{theorem} \label{thm:hs-lsi}
  {\red Let the probability measure $\pi_{N,m,0}$ on $\Omega_{N,m}$ 
  have log-concave generating polynomial, see~\eqref{eq_def_gpi}.}
  Then for any $F: \Omega_{N,m}\to\R$, the measure $\pi=\pi_{N,m,h}$ satisfies
  the modified log-Sobolev inequality (uniformly in $N,m,h$):
  \begin{equation}
    \ent_{\pi}(F) \leq D^{\mathrm{du}}_\pi(F,\log F),
  \end{equation}
  {\red where $D^{\mathrm{du}}_\pi$ is the Dirichlet form \eqref{e:Dirichlet- general} for the down-up walk rates \eqref{e:downup}.}
\end{theorem}
\begin{remark}
The choice~\eqref{e:downup} of normalisation of the jump rate differs from the one in~\cite{MR4203344,2106.04105}, in which~\eqref{e:downup} is further divided by the number $m$ of particles. 
We prefer~\eqref{e:downup} as it corresponds to a $+$ spin typically getting updated once per unit time.
\end{remark}

\begin{remark} \label{rk:HS}
We remark that we use the down-up walk as an ingredient to obtain an optimal constant in our
modified log-Sobolev inequality for the canonical Ising model.
For a nonoptimal rate, one could use the 
log-Sobolev inequalities for $\pi_{N,m,h}$ obtained in \cite[Theorem~3]{MR4564432}
as a consequence of the strong Rayleigh property satisfied by $\pi_{N,m,h}$.
Note that here the lower bound on the log-Sobolev (or modified log-Sobolev) constant is $1/N$
instead of $1$:
For any $F: \Omega_{N,m}\to\R$, the measure $\pi=\pi_{N,m,h}$ satisfies
\begin{equation}
  \ent_{\pi}(F) \leq 2N D_\pi(\sqrt{F}), 
\end{equation}
where the Dirichlet form is the standard Dirichlet form.
The use of this inequality would simplify  the  arguments and yield results without
assumptions on $\max_i \sum_{j\neq i}|A_{ij}|$ or $\max_i |h_i|$ but without
the optimal rate obtained by using the down-up walk.
To recover the optimal rate one could imagine using instead \cite[Theorem~2]{MR4564432}.
However, the rates there are not explicit (though perhaps could be made explicit),
while we need explicit rates for a comparison estimate to recover the original measure.
\end{remark}

\subsection{Proof of Theorem~\ref{thm:mlsi}}\label{sec:thm_mlsi}

Next we collect properties of the renormalised measure $\nu_t$ and the fluctuation measure $\mu_{t}^\varphi$
defined in \eqref{e:nut-def} respectively \eqref{e:mut-def}.
The following convexity is a key observation for the proof of Theorem~\ref{thm:mlsi}.
{\red The proofs of Lemmas~\ref{lem:Vren} and \ref{lem:covbd} only use the covariance bound \eqref{e:var-bd} and not the product structure of the measure directly.}

\begin{lemma} \label{lem:Vren}
For $\alphab>0$ and $\varphi \in X_{N,0}$, the renormalised potential
\begin{equation}
  V_0(\varphi) = - \log \sum_{\sigma\in\Omega_{N,m}} e^{-\frac{\alphab}{2}(\varphi-\sigma,\varphi-\sigma)} \pi_{N,m,h}(\sigma)
  + \mathrm{constant}
\end{equation}
has Hessian on $X_{N,0}$ bounded below by $\alphab-2\alphab^2$.
\end{lemma}

\begin{proof}
For $f: [N]\to \R$ with $\sum_i f_i =0$, the Hessian of $V_0$ is:
\begin{equation}
  \big(f, \He V_0(\varphi) f\big) = \alphab|f|_2^2 - \alphab^2\var_{\pi_{N,m,h+\alphab\varphi}}((f,\sigma)).
\end{equation}
Thus \eqref{e:var-bd} implies $\He V_0(\varphi) \geq \beta -2 \beta^2$ as claimed.
\end{proof}

\begin{lemma} \label{lem:covbd}
  Assume $|A|_{X_{N,0}}<1$ and let $\beta< 1/2$. Then the canonical Ising  {\red measure $\nu_{\beta,h}$} with inverse temperature $\beta$
  and external field $h\in \R^N$ satisfies
  \begin{equation}
    \var_{\nu_{\beta,h}}((f,\sigma)) \leq (\tfrac12-\beta)^{-1}|f|_2.
  \end{equation}
  In particular, $\cov(\mu_{t}^\varphi) \leq (\frac12-t)^{-1}$.
\end{lemma}

\begin{proof}
  Using the decomposition \eqref{eq_decomp_Ising}, with $\nu=\nu_{\beta,h}$ and $\pi=\pi_{N,m,\red h+ \beta \varphi}$,
  \begin{align}
    \var_{\nu}((f,\sigma))
    &= \E_{\nu_0}[\var_\pi((f,\sigma))] + \var_{\nu_0}(\E_{\pi}(f,\sigma))
      \nnb
    &\leq
      2|f|_2^2 + \frac{\alphab^2}{\alphab-2\alphab^2}
      \E_{\nu_0} \Big[|\nabla_h \E_\pi[(f,\sigma)]|_2^2 \Big]
  \end{align}
  where the bound of the first term is \eqref{e:var-bd} and that of the second term
  the Bakry--\'Emery criterion for the spectral gap of $\nu_0$ which is strictly log-concave for $\alphab<1/2$ by  Lemma~\ref{lem:Vren}.
  {\red We also used above that $\nabla_\varphi \E_\pi [(f,\sigma)] = \beta \nabla_h \E_\pi[(f,\sigma)]$.}   
  The second term can be simplified by again using  \eqref{e:var-bd}:
  \begin{align}
    |\nabla_h \E_{\pi}[(f,\sigma)]|_2^2=
    \sup_{|g|_2\leq 1}(g,\nabla_h \E_{\pi}[(f,\sigma)])^2
    =\sup_{|g|_2\leq 1} \cov_{\pi}((f,\sigma),(g,\sigma))^2 \leq 4|f|_2^2.
  \end{align}
  This gives the claim:
  \begin{equation}
    \var_{\nu}((f,\sigma)) \leq
    \pa{2+\frac{4\alphab^2}{\alphab-2\alphab^2}} |f|_2^2
    =
    \pa{\frac{2\alphab}{\alphab-2\alphab^2}} |f|_2^2.
    \qedhere
  \end{equation}
 {\red
The bound on the covariance matrix of $\mu_{t}^\varphi$  follows. Indeed, by the choice \eqref{eq: covariance decomposition} of the covariance $C_t$, the fluctuation measure \eqref{e:mut-def} 
 is a canonical Ising measure with parameter $t$ instead of $\beta$.}
 \end{proof}

The following key estimate
connects the Dirichlet forms of the down-up walk of the {\red infinite temperature measure} 
with that of the {\red finite temperature} canonical Ising model via the decomposition    \eqref{eq_decomp_Ising}.
Its proof is postponed to Section~\ref{sec:downup}.
{\red The proof actually only uses the decomposition \eqref{eq_decomp_Ising} 
and that the measure $\pi_{N,m,0}$ on $\Omega_{N,m}$ is uniform on its support and not the product structure directly.}

\begin{lemma}\label{lemm_bound_DF_no_concavity-bis}
{\red Assume $|A|_{X_{N,0}}< 1$.}
Let $D^{\mathrm{du}}_\pi = D^{\mathrm{du}}_{\pi_{N,m,\red h+\alphab\varphi}}$ be the Dirichlet form with the jump rates of the down-up walk \eqref{e:downup}
associated with the measure $\pi= \pi_{N,m,h+\beta\varphi}$,
and likewise for $D^{\mathrm{du}}_\nu$ {\red with $\nu = \nu_{\beta, h}$ the canonical Ising measure}.
There is $K(\beta)>0$ independent of $N,m,h$ and the test function $F:\Omega_{N,m}\rightarrow\R_+$ such that:
\begin{equation} \label{e:bound_DF_no_concavity}
\E_{\nu_0} [D^{\mathrm{du}}_\pi(F,\log F)]
\leq 
K(\beta) D^{\mathrm{du}}_{\nu}(F,\log F)
.
\end{equation}
The same bound also holds with $D^{\mathrm{du}}_\pi(F),D^{\mathrm{du}}_\nu(F)$ replacing $D^{\mathrm{du}}_{\pi}(F,\log F),D^{\mathrm{du}}_{\nu}(F,\log F)$.
\end{lemma}

\begin{proof}[Proof of Theorem~\ref{thm:mlsi}]
  We first apply the entropic stability estimate from \cite[Proposition 39]{2203.04163}. 
  {\red There the estimate is not formulated for measures with constrained magnetisation, but it is straightforward to
  extend the estimate to this setting as was explicitly done in \cite[Section 3.7]{2307.07619}. The notations below are the same as the latter reference. 
  Writing $\dot C_t$ for the derivative of $C_t$ with respect to $t$,     
  the entropic stability estimate states that (recall  $\mu_t^\varphi$  from \eqref{e:mut-def})}
  \begin{equation} 
  \label{e:entstab}
    \red 2 (\nabla_\varphi \sqrt{\E_{\mu_t^\varphi}[F]})^2_{\dot C_t} \leq \alpha_t\ent_{\mu_t^\varphi}(F),
  \end{equation}
  provided that the numbers $\alpha_t$ satisfy 
  \begin{equation}
    \dot C_t C_t^{-1} \cov(\mu_{t}^\varphi) C_t^{-1} \dot C_t \leq \alpha_t \dot C_t.
  \end{equation}
    Since $\dot C_t C_t^{-2} = (1-A)$ and {\red $A$ has spectrum in $(0,1)$}, this condition follows from $\cov(\mu_{t}^\varphi)\leq \alpha_t$
  as in the covariance condition (CC), where we recall that $\mu_t^\varphi$ is again a canonical Ising measure at inverse temperature $t$.
  For $\beta<1/2$, by Lemma~\ref{lem:covbd}, one can take $\alpha_t = (\frac12 - t)^{-1}$ and thus {\red the constant $\Cbeta$ in the covariance condition (CC) is given by}
  $\Cbeta = \exp(\log(\frac12)-\log(\frac12-\beta))= (1-2\beta)^{-1}$, which is the same as in the spectral condition (SC).

  {\red From \eqref{e:entstab}, it follows that
  \begin{equation}
    \ent_{\nu}(F) \leq \Cbeta \E_{\nu_0}[\ent_{\pi}(F)], \qquad \Cbeta = e^{\int_0^\beta \alpha_t\, dt}.
  \end{equation}
  Indeed,
  \begin{equation}
    \ddp{}{t} \E_{\nu_t}[\ent_{\mu_t^\varphi}(F)] = \E_{\nu_t}\qB{2(\nabla_\varphi \sqrt{\E_{\mu_t^\varphi}[F]})^2_{\dot C_t}}
    \leq \alpha_t\E_{\nu_t}[\ent_{\mu_t^\varphi}(F)]
\label{eq: derivative entropy}
  \end{equation}
  and $\E_{\nu_t}[\ent_{\mu_t^\varphi}(F)] \to \ent_{\mu_\beta^0}(F)=\ent_{\nu}(F)$ as $t\to \beta$.}

  Next we apply the modified log-Sobolev inequality for the down-up walk for $\pi$ from Theorem~\ref{thm:hs-lsi},
  followed by Lemma~\ref{lemm_bound_DF_no_concavity-bis}:
  \begin{equation} \label{e:pf-mlsi-final}
    \ent_\nu(F)
    \leq \Cbeta \E_{\nu_0}[D^{\mathrm{du}}_\pi(F,\log F)]
    \leq \Cbeta
    {\red K(\beta)}D^{\mathrm{du}}_\nu(F,\log F),
  \end{equation}
  where $D^{\mathrm{du}}_\nu$ is the Dirichlet form associated with the rates of the down-up walk \eqref{e:downup}.

  As discussed below \eqref{e:equiv-downup-standard},
  the down-up walk Dirichlet form is equivalent
  to the standard Dirichlet form (which is the one in the statement of Theorem~\ref{thm:mlsi})
  provided that \eqref{e:equiv-downup-standard} holds.
  {\red For the canonical Ising model,}
  this is the case with constants depending on $\bar A$ and $\bar h$.
\end{proof}

{\red
\begin{remark}
As mentioned already in Remark \ref{rem: Glauber}, the measure decomposition \eqref{eq_decomp_Ising_t} has been  used to derive log-Sobolev inequalities for different models.
There are two different strategies to estimate the expectation of $\nabla_\varphi \sqrt{\E_{\mu_t^\varphi}[F]}$ appearing in \eqref{eq: derivative entropy}:
it can either be bounded from above by the Dirichlet form or by the entropy as in \eqref{eq: derivative entropy}. 
We refer to \cite[Section 3.7]{2307.07619} for a discussion of both approaches. 
For the derivation of Theorem~\ref{thm:mlsi}, we relied on the second method introduced in \cite{2203.04163} as recovering the Dirichlet form of the down-up walk dynamics lead to additional complications. 
\end{remark}

\subsection{Proof of Theorem~\ref{thm:matroid}}
\label{sec:pfmatroid}

Essentially the same proof as that of Theorem~\ref{thm:mlsi} also implies
Theorem~\ref{thm:matroid} as follows.
Note that Lemmas~\ref{lem:Vren} and \ref{lem:covbd} only use the covariance bound \eqref{e:var-bd} and not
the product stucture of $\pi_{N,m,h}$ directly. Hence they apply to all negatively correlated measures.
More generally, if the generating polynomial of $\pi = \pi_{N,m,h}$ is log-concave then \eqref{e:var-bd} holds with a factor
$4$ instead of $2$ and the lemmas continue to hold in this setting (with the replacement of $2$ by $4$). Indeed,
by definition of the generating polynomial  \eqref{eq_def_gpi},
\begin{align}
  \partial_{z_i}\partial_{z_j} \log g_\pi(1,\dots, 1) &= \P_\pi[i,j\in I] - \P_\pi[i\in I]\P_\pi[j\in I] \qquad (i\neq j)
  \\
  \partial_{z_i}^2 \log g_\pi(1, \dots, 1) &= - \P_\pi[i\in I]^2.
\end{align}
Since $1_{i\in I} = \frac12(1+\sigma_i)$, this means that
\begin{align}
  \partial_{z_i}\partial_{z_j} \log g_\pi(1,\dots, 1) &= \frac14 \cov_\pi(1+\sigma_i,1+\sigma_j) = \frac14 \cov_\pi(\sigma_i,\sigma_j)
  \\
  \partial_{z_i}^2 \log g_\pi(1,\dots, 1) &= -\frac14 (1+\E_\pi[\sigma_i])^2 
  = -\frac12 (1+\E_\pi[\sigma_i]) + \frac14\var_\pi(\sigma_i).
\end{align}
Thus the log-concavity $\He \log g_\pi \leq 0$ implies $(\cov_\pi (\sigma_i,\sigma_j) - 2 (1+\E_\pi [\sigma_i]))_{i,j} \leq 0$ as quadratic forms, i.e.,
\begin{equation}
  \var_\pi((f,\sigma)) \leq 2 \sum_i f_i^2 (1+\E_\pi [\sigma_i]) \leq 4 |f|_2^2.
\end{equation}
Lemma~\ref{lemm_bound_DF_no_concavity-bis} only uses the decomposition
$\E_{\nu} = \E_{\nu_0}[\E_{\pi_{N,m,\beta\varphi}}]$ and that $\pi_{N,m,0}$ is uniform on its support.
Thus it holds if $\pi_{N,m,0}$ is the uniform measure on the bases of a matroid.
We summarise this as follows.

\begin{lemma}
Lemmas~\ref{lem:Vren} and \ref{lem:covbd} continue to hold if $\pi_{N,m,0}$ is the uniform
measure on the bases of a matroid, with the factor $2$  replaced by $4$,
and Lemma~\ref{lemm_bound_DF_no_concavity-bis} continues to hold without change.
\end{lemma}

\begin{proof}[Proof of Theorem~\ref{thm:matroid}]
By the above discussion, and since the entropic stability property \eqref{e:entstab} does not use any properties of $\mu_0^\varphi$
except that   \eqref{e:mut-def} holds on its support, the proof of \eqref{e:pf-mlsi-final} continues to hold.
To obtain the statement as formuled in the theorem, it remains to convert the Ising variables to occupation variables
$\omega_e = \frac12 (\sigma_e +1)$. Note that
\begin{equation}
  \prod_{e,f} (1+\epsilon_{ef} \omega_e \omega_f)
  =
  \prod_{e,f} \Big(1+\frac{\epsilon_{ef}}{4} (\sigma_e+1)(\sigma_f+1)\Big)
   \propto e^{-\frac{\beta}{2} \sum_{e,f} A_{ef} \sigma_e\sigma_f + \sum_e h_e\sigma_e}
\end{equation}
with $\beta A_{ef} = -\frac{1}{2}\log(1+\epsilon_{ef}) = O(\epsilon_{ef})$  and some $h_e \in \R$ depending on the $(\epsilon_{ef})$,
where we used that
\begin{equation}
\log\Big(1+\frac{\epsilon_{ef}}{4}(\sigma_e+1)(\sigma_f+1)\Big)
=
\frac{(\sigma_e+1)(\sigma_f+1)}{4}\log(1+\epsilon_{ef})
.
\end{equation}
By adding a multiple of the identity and normalising, we can again assume that $A$ has spectrum in $(0,1)$
and the condition on the $(\epsilon_{ef})$ then implies that $\beta = O(\bar\epsilon)$ is sufficiently small so that $\beta<\frac14$.
\end{proof}
}

\section{Recovering the Dirichlet form of the down-up walk}
\label{sec:downup}

In this section, we prove Lemma~\ref{lemm_bound_DF_no_concavity-bis} which we restate below as Lemma~\ref{lemm_bound_DF_no_concavity}
for convenience.
Also recall that the jump rates of the down-up walk for a measure $\pi$ on $\Omega_{N,m}$ with $m\leq 0$ are given by:
\begin{equation}
{\red c^{\mathrm{du}}}(\sigma,\sigma^{ij}) 
=
{\red c^{\mathrm{du}}_{\pi}} (\sigma,\sigma^{ij}) 
=
{\bf 1}_{\sigma_i=1}{\bf 1}_{\sigma_j=-1}\frac{\pi(\sigma^{ij})}{\sum_{\red k\in J_i(\sigma)} \pi(\sigma^{ik})}
,
\label{eq_jump_rates_CGM}
\end{equation}
{\red with $J_i(\sigma)$ the set of holes ($-$ spins) where a particle ($+$ spin) can be added after removing the one at $i$.}
If instead $m\geq 0$, 
the dynamics to consider has the jump rates 
of the up-down walk given by considering the above choice for $-\sigma$.

\begin{lemma}\label{lemm_bound_DF_no_concavity}
{\red In the setting of Section~\ref{sec_def_nu_r_pi},
assume $|A|_{X_{N,0}}< 1$, and that
the measure $\pi_{N,m,0}$ on $\Omega_{N,m}$ is uniform on its support (see Remark~\ref{rk:decomp-matroid}), and}  
let $D^{\mathrm{du}}_\pi$ be the Dirichlet form with the jump rates of the down-up walk \eqref{eq_jump_rates_CGM}
associated with the measure $\pi= \pi_{N,m,h+\beta\varphi}$,
and likewise for $D^{\mathrm{du}}_\nu$.
There is $K(\beta)>0$ independent of $N,m,h$ and the test function $F:\Omega_{N,m}\rightarrow\R_+$ such that:
\begin{equation} \label{e:bound_DF_no_concavity-bis}
\E_{\nu_0}[D^{\mathrm{du}}_{\pi}(F)]
\leq 
K(\beta) D^{\mathrm{du}}_{\nu}(F)
=
K(\beta) \sum_{\sigma,\sigma'\in\Omega_{N,m}}\nu(\sigma) c_\nu(\sigma,\sigma') \big[F(\sigma')-F(\sigma)\big]^2.
\end{equation}
The same bound also holds with $D^{\mathrm{du}}_\pi(F),D^{\mathrm{du}}_\nu(F)$ replaced by $D^{\mathrm{du}}_{\pi}(F,\log F),D^{\mathrm{du}}_{\nu}(F,\log F)$.
\end{lemma}
\begin{proof}
Fix $F:\Omega_{N,m}\rightarrow\R_+$. 
{\red Since $\pi_{N,m,0}$ is uniform on its support, $\pi_{N,m,0}(\sigma)=\pi_{N,m,0}(\sigma')$ for all $\sigma,\sigma' \in \{\pm 1\}^N$ for which both sides are nonzero.
Hence}
the jump rates for $\pi=\pi_{N,m,h+\alphab\varphi}$ read:
\begin{align}
{\red c^{\mathrm{du}}_\pi} (\sigma,\sigma^{ij}) 
&=
\red {\bf 1}_{\sigma_i=1}{\bf 1}_{\sigma_j=-1}\frac{e^{(\sigma_i-\sigma_j)\big(\alphab(\varphi_j-\varphi_i) + h_j-h_i\big)} \pi_{N,m,0}(\sigma^{ij})}{\sum_{k \in J_i(\sigma) } e^{(\sigma_i-\sigma_k)\big(\alphab(\varphi_k-\varphi_i)+h_k-h_i\big)}\pi_{N,m,0}(\sigma^{ik})}
\nnb
&=
{\bf 1}_{\sigma_i=1}{\bf 1}_{\sigma_j=-1} {\red {\bf 1}_{j\in J_i(\sigma)}}\frac{e^{(\sigma_i-\sigma_j)\big(\alphab(\varphi_j-\varphi_i) + h_j-h_i\big)}}{\sum_{\red k \in J_i(\sigma) } e^{(\sigma_i-\sigma_k)\big(\alphab(\varphi_k-\varphi_i)+h_k-h_i\big)}}
\nnb
&= 
{\bf 1}_{\sigma_i=1}{\bf 1}_{\sigma_j=-1} {\red {\bf 1}_{j\in J_i(\sigma)}}\frac{e^{2(\alphab\varphi_j + h_j)}}{\sum_{\red k \in J_i(\sigma) } e^{2(\alphab\varphi_k + h_k)}}
.
\end{align}
It suffices to prove that,
for each $\sigma\in\Omega_{N,m}$, each $i\in I(\sigma)$, and each
$\red j \in J_i(\sigma)$, all henceforth fixed:
\begin{equation}
\E_{\nu_0}\big[\pi(\sigma) {\red c^{\mathrm{du}}_\pi} (\sigma,\sigma^{ij})\big]
\leq 
K\nu(\sigma)  {\red c^{\mathrm{du}}_\nu} (\sigma,\sigma^{ij})
\label{eq_bound_jump_rates_without_concavity}
\end{equation}
with a constant $K=K(\beta)$.
(We remark that if $\pi(\sigma) {\red c^{\mathrm{du}}_\pi} (\sigma,\sigma^{ij})$  were a concave function of $\pi$, 
then this inequality would hold with constant $K=1$. 
This is unfortunately not the case.)

Recall that $\E_{\nu_0}[\pi(\sigma)] = \nu(\sigma)$ by definition.
From the expressions for $\nu_0$ and $\pi_{N,m,h+\alphab\varphi}$ in Section~\ref{sec_def_nu_r_pi},
{\red with $C = C_\beta-C_0 = (\beta A)^{-1} - \beta^{-1}P$,}
the left-hand side of~\eqref{eq_bound_jump_rates_without_concavity} reads:
\begin{align}
\E_{\nu_0}\big[\pi(\sigma) {\red c^{\mathrm{du}}_\pi} (\sigma,\sigma^{ij})\big]
&\propto
  \int_{X_{N,0}} \exp\bigg[-\frac{(\varphi,C^{-1}\varphi)}{2} - \frac{\alphab(\varphi-\sigma,\varphi-\sigma)}{2} + (h,\sigma)\bigg] 
\nonumber\\
&\hspace{5cm}\times\frac{e^{2(\alphab\varphi_j + h_j)}}{\sum_{\red k \in J_i(\sigma) } e^{2(\alphab\varphi_k + h_k)}}\, d\varphi
\nnb
&=
\nu(\sigma){\bf E}_{\gamma}\bigg[\frac{e^{2(\alphab\varphi_j + h_j)}}{\sum_{\red k\in J_i(\sigma)} e^{2(\alphab\varphi_k + h_k)}}\bigg]
,
\end{align}
where ${\bf E}_\gamma$ is the expectation with respect to a Gaussian measure $\gamma$ on $X_{N,0}$ with covariance matrix:
\begin{equation}\label{e:C-alpha}
\red  (C^{-1}+\alphab P)^{-1} = \frac{1}{\alphab} (P-A)
\end{equation}
and mean:
\begin{equation} \label{e:a-sigma}
a^\sigma
=
\alphab(C^{-1}+\alphab P)^{-1}\sigma
= \sigma - A\sigma
.
\end{equation}
{\red For the last equalities in \eqref{e:C-alpha} and \eqref{e:a-sigma},}
from $C=(\beta A)^{-1}-\alphab^{-1} P$ 
where $|A|_{X_{N,0}}<1$ is the operator norm of $A$ acting on $X_{N,0}$
and $P$ is the identity on $X_{N,0}$, we get
\begin{equation}
(\alphab C)^{-1} 
=
(A^{-1}-P)^{-1}
= \frac{A}{P-A} 
\end{equation}
as operators on $X_{N,0}$ {\red so that} 
\begin{equation} \label{e:C-alpha-bis}
\alphab(C^{-1}+\alphab P)^{-1}
= \qa{ \frac{A + (P-A)}{P-A} }^{-1}
= P - A
.
\end{equation}

Writing $\gamma_0$ for the centred Gaussian distribution with covariance $(C^{-1}+\alphab P)^{-1}$, 
we get:
\begin{align}
\label{eq: reecriture gamma0}
\E_{\nu_0}\big[\pi(\sigma) {\red c^{\mathrm{du}}_\pi} (\sigma,\sigma^{ij})\big]
=
\nu(\sigma){\bf E}_{\gamma_0}\Big[\frac{e^{2\alphab(\psi_j + a^\sigma_j) + 2h_j}}{\sum_{\red k \in J_i(\sigma)} e^{2\alphab(\psi_k+a^\sigma_k) + 2h_k}}\Big]
.
\end{align}
To bound this last expectation by ${\red c^{\mathrm{du}}_\nu} (\sigma,\sigma^{ij})$, we first rewrite ${\red c^{\mathrm{du}}_\nu} (\sigma,\sigma^{ij})$.
By~\eqref{eq_jump_rates_CGM}, one has:
\begin{equation}
{\red c^{\mathrm{du}}_\nu} (\sigma,\sigma^{ij})
=
{\bf 1}_{\sigma_i=1}{\bf 1}_{\sigma_j=-1} {\red {\bf 1}_{j\in J_i(\sigma)}}
\frac{\exp\Big[2h_j-\frac{\beta}{2}\nabla_{ij}(\sigma,A\sigma)\Big]}{\sum_{\red k\in J_i(\sigma) }\exp\Big[2h_k-\frac{\beta}{2}\nabla_{ik}(\sigma,A\sigma)\Big]}
,
\end{equation}
where $\nabla_{ij}F(\sigma) = F(\sigma^{ij})-F(\sigma)$. Therefore, defining
\begin{equation}
  c_{k} = \frac{\beta}{2}\nabla_{ij} (\sigma,A\sigma)- \frac{\beta}{2}\nabla_{ik} (\sigma,A\sigma),
\end{equation}
one has
\begin{equation}
{\red c^{\mathrm{du}}_\nu} (\sigma,\sigma^{ij})
= 
\frac{e^{2h_j}}{\sum_{\red k\in J_i(\sigma)}e^{2h_k + c_{k}}}
.
\end{equation}
Let $U_h=U_h^{i,j}$ denote the following probability measure on $\red J_i(\sigma)$: 
\begin{equation}
U_h(k)
=
\frac{1}{Z_h}e^{2h_k + c_k}
,\qquad 
Z_h
=
Z_h(\sigma)
=
\sum_{\red k \in J_i(\sigma)} e^{2h_k + c_k}
.
\end{equation}
Then, multiplying and dividing the right-hand side of \eqref{eq: reecriture gamma0} by ${\red c^{\mathrm{du}}_\nu} (\sigma,\sigma^{ij})$ and using Jensen's inequality for the convex function $x\in(0,\infty)\mapsto 1/x$:
\begin{align}
\E_{\nu_0}\big[\pi(\sigma) {\red c^{\mathrm{du}}_\pi} (\sigma,\sigma^{ij})\big]
&=
  \nu(\sigma)  {\red c^{\mathrm{du}}_\nu} (\sigma,\sigma^{ij})
{\bf E}_{\gamma_0}\Big[\frac{e^{2\alphab(\psi_j + a^\sigma_j) }}{\E_{U_h}\big[ e^{2\alphab(\psi_\cdot+a^\sigma_\cdot)  - c_{\cdot }  }\big]}\Big]
\nonumber\\
&\leq
\nu(\sigma) {\red c^{\mathrm{du}}_\nu} (\sigma,\sigma^{ij}) 
\E_{U_h}\bigg[{\bf E}_{\gamma_0}\Big[e^{2\alphab(\psi_j -\psi_\cdot+ a^\sigma_j-a^\sigma_\cdot)  }\Big] e^{ c_{\cdot}}\bigg]
.
\label{eq_bound_average_jump_rates}
\end{align}
Thus \eqref{eq_bound_jump_rates_without_concavity} holds with
\begin{equation}
  K = \max_{i,j}\E_{U_h}\bigg[{\bf E}_{\gamma_0}\Big[e^{2\alphab(\psi_j -\psi_\cdot+ a^\sigma_j-a^\sigma_\cdot)  }\Big] e^{   c_{\cdot}}\bigg].
  \label{eq_def_K}
\end{equation}
Let us check that $K$ is bounded above uniformly in $\sigma$ in terms of $\beta$ only.
For each $\red k \in J_i(\sigma)$, the Gaussian expectation on the right-hand side is:
\begin{equation}
{\bf E}_{\gamma_0}\Big[e^{2\alphab(\psi_j -\psi_k+ a^\sigma_j-a^\sigma_k) }\Big]
=
e^{2\alphab(a^\sigma_j-a^\sigma_k)}e^{2\alphab^2 \big(\delta_j-\delta_k,(C^{-1}+\alphab P)^{-1}(\delta_j-\delta_k)\big)}
.
\label{eq_bound_average_gamma_0}
\end{equation}
By \eqref{e:C-alpha} and \eqref{e:a-sigma}, the two terms in the exponent can be written as
\begin{align}
  2\alphab^2 \big(\delta_j-\delta_k,(C^{-1}+\alphab P)^{-1}(\delta_j-\delta_k)\big)
  &=
  2 \alphab \big(\delta_j-\delta_k,(P - A)(\delta_j-\delta_k)\big)
  \label{eq_formula_diff_asigma_line1}
\\
  2\alphab(a^\sigma_j-a^\sigma_k) &= 2 \alphab (\sigma, (P- A)(\delta_j-\delta_k))
  .
  \label{eq_formula_diff_asigma}
\end{align}
{\red
Together with the $c_k$ term in \eqref{eq_def_K}, split the second line as:
\begin{equation} \label{e:ac}
2\alphab(a^\sigma_j-a^\sigma_k) + c_k
=
\Big[2\alphab(a^\sigma_j-a^\sigma_i) +\frac{\beta}{2}\nabla_{ij}(\sigma,A\sigma)\Big] 
- \Big[2\alphab(a^\sigma_k-a^\sigma_i)+\frac{\beta}{2}\nabla_{ik}(\sigma,A\sigma)\Big].
\end{equation}
Since $\sigma_j=-1$ and $\sigma_i=1$, one has:
\begin{align}
\frac{\beta}{2}\nabla_{ij}(\sigma,A\sigma)
&=
-\beta\sum_{\ell\notin\{i,j\}}\sigma_\ell (\sigma_i-\sigma_j)(A_{i\ell}-A_{j\ell}),
\nnb
&=
-2\beta\sum_{\ell\notin\{i,j\}}\sigma_\ell(A_{i\ell}-A_{j\ell})
\nnb
&=
	-2\beta \big(\sigma, A(\delta_i-\delta_j)\big) +2\beta(A_{ii} -A_{ji} - A_{ij} + A_{jj})
.
\end{align}
Comparing with~\eqref{eq_formula_diff_asigma} and using that $\big(\sigma, P (\delta_i-\delta_j)\big)  = \big(\delta_i-\delta_j, P (\delta_i-\delta_j)\big)$ as 
$\sigma_j=-1$ and $\sigma_i=1$, we get:
\begin{align}
2\beta (a^\sigma_j-a^\sigma_i) + 
\frac{\beta}{2}\nabla_{ij}(\sigma,A\sigma)
&=
 - 2\beta \big(\sigma, P (\delta_i-\delta_j)\big) 
+ 2\beta \big(\delta_i-\delta_j, A(\delta_i-\delta_j)\big) 
\nnb
&=
-2\beta \big(\delta_i-\delta_j, (P-A)(\delta_i-\delta_j)\big) .
\end{align}
If $k=i$ then combining with~\eqref{eq_formula_diff_asigma_line1} gives:
\begin{equation}
 2 \alphab \big(\delta_j-\delta_k,(P - A)(\delta_j-\delta_k)\big) + 2\beta (a^\sigma_j-a^\sigma_k) + 
c_k
=
0
.
\end{equation}
If $k\neq i$ then $\sigma_k=-1=\sigma_j$ and both brackets in~\eqref{e:ac} are similar. 
Thus:
\begin{equation}
2\beta (a^\sigma_j-a^\sigma_k) + 
c_k
=
-2\beta\big( \delta_i-\delta_j,(P-A)(\delta_i-\delta_j)\big)
+2\beta\big( \delta_i-\delta_k,(P-A)(\delta_i-\delta_k)\big)
.
\end{equation}
Since $0<A<P$, 
combining with~\eqref{eq_formula_diff_asigma_line1} gives
\begin{equation}
2 \alphab \big(\delta_j-\delta_k,(P - A)(\delta_j-\delta_k)\big) + 2\beta (a^\sigma_j-a^\sigma_k) + 
c_k
\leq 
8\beta
.
\end{equation}
Recalling~\eqref{eq_def_K} we conclude $K\leq e^{8\beta}$.}
\end{proof}

\section{Moving particle lemmas}
\label{sec:mp}

The following lemma was proved in \cite[Section 5]{MR2271489} and \cite[Theorem 6.1]{MR1621569}.

\begin{lemma}[Moving particle lemma] 
  \label{lem:mp}
  Let $\Lambda \subset \Z$ be an interval and assume the canonical Ising model $\nu$ on $\Lambda$
  has bounded couplings $\max_i\sum_{j \neq i} |A_{ij}|\leq \bar A$ and fields $\max_i |h_i|\leq \bar h$
  and that it satisfies the finite range property $A_{ij}=0$ if $\dist(i,j)>R$ for some $R < \infty$.
  Then for $i<j$ with $i,j \in\Lambda$,
  \begin{equation}
    \E_{\nu}[(F(\sigma)-F(\sigma^{ij}))^2] \leq C(\beta \bar A,\bar h,R)|i-j| \sum_{k=i}^{j-1} \E_{\nu}[(F(\sigma)-F(\sigma^{k,k+1}))^2].
  \end{equation}
\end{lemma}

\begin{proof}
  It suffices to consider the case that $j=i_\ell$ where $i_k=i+k(R+1)$ and show that
  \begin{equation} \label{e:pf-mp-1}
    \E_{\nu}[(F(\sigma)-F(\sigma^{ij}))^2] \leq C(\beta \bar A, \bar h,R)|\ell| \sum_{k=i}^{\ell-1} \E_{\nu}[(F(\sigma)-F(\sigma^{i_k,i_{k+1}}))^2].
  \end{equation}
  The general case then follows by bounding (with a different constant)
  \begin{equation}
    \E_{\nu}[(F(\sigma)-F(\sigma^{i_k,i_{k+1}}))^2]
    \leq
    C(\beta \bar A, \bar h, R) \sum_{m=0}^{R} \E_{\nu}[(F(\sigma)-F(\sigma^{i+m,i+m+1}))^2],
  \end{equation}
  whose proof we omit (and which is easy since $R$ is fixed).
  To verify \eqref{e:pf-mp-1}, let $I = \{i_0, \dots, i_\ell\}$.
  By the finite-range property, conditional on the spins outside $I$, the canonical Ising measure restricted to $\sigma|_I$ is a product Bernoulli measure
  conditioned on its sum  with marginal probabilities $\propto e^{\alpha_k \sigma_{i_k}}$
  where $\alpha_k \in [-B,B]$ for $B = \bar A + \bar h$.
  Hence by \cite[Lemma~5.2]{MR2271489}, the inequality  \eqref{e:pf-mp-1} follows.
\end{proof}

A standard consequence of the above moving particle lemma is the following comparison estimate for the Kawasaki Dirichlet
form on cubes $\Lambda \subset \Z^d$.

\begin{corollary} \label{cor:mp-Zd}
  Let $d\geq 1$. For the canonical Ising model on a cube $\Lambda \subset \Z^d$ of side length $L$ {\red under the assumptions $\max_i\sum_j|A_{ij}|\leq \bar A$, $\max_i|h_i|\leq \bar h$, and $A_{ij}=0$ if $\dist(i,j) > R$},
  \begin{equation}
    \frac{1}{L^d} \sum_{i,j}\E_{\nu}[(F(\sigma)-F(\sigma^{ij}))^2] \leq C(d,\beta\bar A,\bar h,R)L^2 \sum_{i \sim j} \E_{\nu}[(F(\sigma)-F(\sigma^{ij}))^2],
  \end{equation}
  where the sum on the left runs over all vertices $i$ and $j$ and the sum on the right over pairs of nearest neighbours.
\end{corollary}

\begin{proof}
  This is a standard consequence which we include for completeness and in preparation for the argument on random regular graphs.
  For each $i,j\in\Lambda$, let $\gamma_{ij}$ be the nearest-neighbour path from $i$ to $j$ that moves first in the first coordinate
  direction, then in the second, and so on.
  By conditioning on the spins outside this path,
  Lemma~\ref{lem:mp} can be applied to the one-dimensional Ising model along the path (which has the same range with respect to the graph distance).
  This gives
  \begin{equation}
    \E_{\nu}[(F(\sigma)-F(\sigma^{ij}))^2] \leq C(d,\beta\bar A,\bar h, R)|\gamma_{ij}| \sum_{k\ell \in\gamma_{ij}} \E_{\nu}[(F(\sigma)-F(\sigma^{k\ell}))^2],
  \end{equation}
  and  the assertion follows from the simple path counting argument
  \begin{equation}
    \max_e \frac{1}{L^d}\sum_{i,j} |\gamma_{ij}| 1_{e\in \gamma_{ij}} \leq C(d)L^2
  \end{equation}
  where $e$ runs over the nearest-neighbour edges of $\Lambda$.
\end{proof}

The next corollary applies the moving particle lemma to compare Dirichlet forms for the canonical Ising model
on a random regular graph.
We expect that the logarithmic factor in the statement is technical
since,
{\red for instance in the nearest-neighbour case,} the adjacency matrix of the random regular graph has a uniform spectral gap.
Our argument below effectively amounts to estimating this spectral gap using the path counting method which is too crude to
obtain the optimal estimate. 

\begin{corollary} \label{cor:mp-rrg}
  Let $d\geq 3$.  For the canonical Ising model on a random $d$-regular graph on $N$ vertices {\red under the assumptions $\max_i\sum_j|A_{ij}|\leq \bar A$, $\max_i|h_i|\leq \bar h$, and $A_{ij}=0$ if $\dist(i,j)>R$},
  \begin{equation}
    \frac{1}{N} \sum_{i,j}\E_{\nu}[(F(\sigma)-F(\sigma^{ij}))^2] \leq C(d,{\red \beta\bar A,\bar h, R})\log^4\!N \sum_{i \sim j} \E_{\nu}[(F(\sigma)-F(\sigma^{ij}))^2],
  \end{equation}
  with high probability under the randomness of the random regular graph. 
  The sum on the left again runs over all vertices $i$ and $j$ while the sum on the right runs over pairs of nearest neighbours
  {\red in the random regular graph.}
\end{corollary}

\begin{proof}
  For each pair of vertices $i,j \in [N]$, fix some geodesic nearest-neighbour path $\gamma_{ij}$ of length $|\gamma_{ij}|$ between $i$ and $j$.
  In particular, $\gamma_{ij} \neq \gamma_{k\ell}$ if $\{i,j\}\neq \{k,\ell\}$
  and the canonical Ising model conditioned on the spins outside $\gamma_{ij}$ again has the same range 
  as the original one since $\gamma_{ij}$ is a geodesic.
  Therefore it suffices to bound
  \begin{equation}
    \max_e \frac{1}{2N} \sum_{i,j} |\gamma_{ij}| 1_{e\in \gamma_{ij}}.
  \end{equation}
  For every edge $e$, one has the trivial bound
  \begin{equation}
    \frac12 \sum_{i,j} 1_{e\in\gamma_{ij}} \leq \text{number of simple paths containing $e$ of length at most $D$},
  \end{equation}
  where $D$ is the diameter of the graph.
  It is known \cite[Theorem 10.14]{MR1864966} that the diameter of a random regular graph satisfies (again with high probability)
  \begin{equation} \label{e:rrg-diameter}
    D \leq \log_{d-1} N + \log_{d-1} \log N + C.
  \end{equation}
  There are at most $(d-1)^k$ simple paths of length $k+1$ containing $e$ at the $m$-th position, $1\leq m \leq k$.
  Therefore there are at most $k(d-1)^k$ simple paths of length $k+1$ containing $e$ and the above sum is
  crudely bounded by
  \begin{equation}
    \sum_{k=1}^{D-1} k(d-1)^k \leq D^2 (d-1)^{D}.
  \end{equation}
  Using \eqref{e:rrg-diameter} gives
  \begin{equation}
    \frac{1}{2N} \sum_{i,j} |\gamma_{ij}| 1_{e\in \gamma_{ij}}
    \leq
    \frac{D}{2N} \sum_{i,j} 1_{e\in \gamma_{ij}}
    \leq
    \frac{1}{N} D^3 (d-1)^D \leq C_d \log^4\!N
  \end{equation}
  and the claim follows from Lemma~\ref{lem:mp}.
\end{proof}

\appendix 
\section{Proof of Theorem~\ref{thm:hs-lsi}}
\label{app:downup}

In this appendix, we include an alternative version of the proof of Theorem~\ref{thm:hs-lsi}, 
i.e., we prove:
\begin{equation}
\ent_{\pi}(F)
\leq 
D^{\mathrm{du}}_{\pi}(F,\log F),
\qquad 
F:\Omega_{N,m}\to\R_+
,
\label{eq_mLSI_appendix}
\end{equation}
{\red where $\pi$ is any probability measure on $\Omega_{N,m}$ with log-concave polynomial $g_{\pi}$, see Proposition~\ref{prop_log_concave}.
Above, }$D^{\mathrm{du}}_\pi$ is the Dirichlet form of the bases-exchange dynamics (down-up walk) defined in \eqref{e:downup}.  
It is convenient to see $+$ spins as particles: we write $\Omega_{N,k}$ for $\Omega_{N,m}$ with $k=N(1+m)/2$ and
\begin{equation} \label{e:app-I-def}
  I{\red(\sigma)} =\{i \in [N]: \sigma_i = +1\}.
\end{equation}

Our proof below uses the same entropy contraction argument as in~\cite{MR4203344,2106.04105}. 
A main difference in our presentation 
is that we decouple this entropy contraction argument from the specific dynamics. 
In particular, we explain how to deduce log-Sobolev inequalities (rather than modified log-Sobolev) and modified log-Sobolev inequalities with respect to other dynamics.

\subsection{Entropic independence estimate}

The key ingredient 
is the following contraction estimate, referred to as entropic independence \cite{2106.04105}:
For a probability measure $\pi$ on $\Omega_{N,k}$, introduce the probability vector in $[0,1]^N$
(equivalently viewed as a probability distribution on $\Omega_{N,1}$):
\begin{equation}
\pi D_{k\to 1}
= \frac{1}{k}(\E_\pi[{\bf 1}_{\sigma_i=1}])_{i\in[N]}
.
\label{eq_def_D_kto1}
\end{equation}
The $1$-entropic independence of $\pi$ is the property that, for any probability measure $\mu$ on $\Omega_{N,k}$:
\begin{equation}
\mathbb H\big(\mu D_{k\to 1}|\pi D_{k\to 1})
\leq 
\frac{1}{k} \mathbb H\big(\mu |\pi )
,
\label{eq_entropic_independence_appendix}
\end{equation}
where $\mathbb H(\mu_1|\mu_2)$ is the relative entropy 
between two probability measures $\mu_1,\mu_2$.

\begin{proposition}\label{prop_log_concave}
  If $g_\pi(z)=\E_\pi[z^I]$  is log-concave on $(0,\infty)^N$ then \eqref{eq_entropic_independence_appendix} holds (where $z^I = \prod_{i\in I}z_i$
  and $I$ was defined in \eqref{e:app-I-def}).
\end{proposition}

\begin{proof}
  This is shown in \cite[Section 3]{2106.04105}.
  Let $p=\pi D_{k\to 1}$ and $q=\mu D_{k\to 1}$.
  Then log-concavity implies
  \begin{equation}
    \frac{1}{k}\log g_\pi(z) \leq \log \qa{\sum_i p_i z_i}.
  \end{equation}
  The entropy inequality with $e^G=z^I$ implies, for any $z\in (0,\infty)^N$,
  \begin{equation}
    \bbH(\mu|\pi) =\sup_G \Big\{ \E_\mu[G]-\log \E_\pi[e^{G}] \Big\}
    \geq \E_\mu[\log z^I]  - \log \E_\pi[z^I]
    = k \sum_i q_i \log z_i - \log  g_\pi(z).
  \end{equation}
  Combining both estimates with the choice $z_i=q_i/p_i$ yields the claim:
  \begin{equation}
    \bbH(\mu|\pi) \geq
    k \sum_i q_i \log z_i  - k \log \Big(\sum_i p_iz_i\Big)
    =
    k\sum_i q_i \log (q_i/p_i) = k\bbH(q|p).\qedhere
  \end{equation}
\end{proof}

\subsection{Entropy contraction}

We decompose $\pi$ in terms of the number of particles as follows. 
For each $0\leq \ell\leq k$, 
{\red define $\Omega^\pi_{k-\ell}$ as the set of all configurations with $k-\ell$ particles that are compatible with $\pi$:
\begin{equation}
\Omega^\pi_{k-\ell}
=
\Big\{\varphi\in \Omega_{N,k-\ell}: \text{ there is }\sigma\in\text{supp}(\pi)\subset\Omega_{N,k}\text{ such that }\sigma\geq \varphi\Big\}
,
\end{equation}
where $\sigma\geq \varphi$ means $\sigma_i\geq \varphi_i$ for each $i\in[N]$. 
Define then $\nu_\ell$ as the following probability measure on $\varphi \in \Omega^\pi_{k-\ell}$:}
\begin{equation}
\nu_\ell(\varphi)
=
\frac{1}{\binom{k}{\ell}}\pi(\sigma\geq \varphi)
.
\end{equation}
{\red Note that $\nu_{\ell}$ is indeed a probability measure, 
as exactly $\binom{k}{\ell}$ different $\varphi\in\Omega^\pi_{k-\ell}$ satisfy $\sigma\geq \varphi$ for a given $\sigma\in\text{supp}(\pi)$, 
corresponding to all configurations obtained by removing $\ell$ particles from $\sigma$.}
Let also $\mu_\ell=\mu_\ell^\varphi$ be the measure $\pi$ on $\Omega_{N,k}$ conditioned on $\sigma \geq \varphi$:
\begin{equation}
\mu_\ell^\varphi(\sigma)
=
\pi(\sigma | \sigma \geq \varphi)
.
\end{equation}
Then we have  a decomposition analogous to \eqref{eq_decomp_Ising_t}:
for any $0\leq \ell \leq k$,
\begin{equation} \label{eq_decomp_pi}
  \E_\pi[F]=\E_{\nu_\ell}[\E_{\mu_\ell^\varphi}[F]].
\end{equation}
Let $P_{\ell,\ell+1}= P_{\ell,\ell+1}(\psi,\cdot)$ be the probability measure on $\Omega^\pi_{k-\ell}$ that adds a particle to $\psi \in\Omega^\pi_{k-(\ell+1)}$ 
{\red in a way compatible with $\pi$}:
\begin{equation}
  \E_{\nu_{\ell}}[F(\varphi)] = \E_{\nu_{\ell+1}}[P_{\ell,\ell+1}F(\psi)],
  \qquad
  P_{\ell,\ell+1}\E_{\mu_\ell}[F]
  =
  \E_{\mu_{\ell+1}}[F]
  .
  \label{eq_prop_P_ell_ell+1}
\end{equation}
Explicitly, for $\psi \in \Omega^\pi_{k-  (\ell+1) }$:
\begin{equation}
  P_{\ell,\ell+1}F(\psi) = \frac{1}{\ell+1}\sum_{i\in[N]} \frac{\pi(\sigma\geq \psi+{\bf 1}_i)}{\pi(\sigma\geq \psi)} 1_{{\red \psi + {\bf 1}_i\in \Omega^\pi_{k-\ell}}} F(\psi+1_i)
  ,
  \label{eq_def_P_ell_ell+1}
\end{equation}
and $P_{\ell,\ell+1}(\psi,\cdot)$ can equivalently be seen as a probability measure $(P^\psi_{\ell,\ell+1}(i))_i$ on $[N]$: 
\begin{equation}
P^\psi_{\ell,\ell+1}(i)
=
P_{\ell,\ell+1}(\psi,\psi+{\bf 1}_i)
=
{\bf 1}_{{\red \psi + {\bf 1}_i\in \Omega^\pi_{k-\ell}}}\frac{1}{\ell+1}\frac{\pi(\sigma\geq \psi+{\bf 1}_i)}{\pi(\sigma\geq \psi)} 
.
\end{equation}
{\red E.g. the right-hand side of Equation~\eqref{eq_prop_P_ell_ell+1} then indeed holds since:}
\begin{align}
P_{\ell,\ell+1}\E_{\mu_{\ell}}[F]
&:=
\frac{1}{\ell+1}\sum_{j\in [N]}\frac{\pi(\sigma\geq \psi+{\bf 1}_j)}{\pi(\sigma\geq \psi)}{\bf 1}_{{\red \psi + {\bf 1}_j\in \Omega^\pi_{k-\ell}}}\E_{\pi(\cdot|\sigma\geq\psi+{\bf 1}_j)}[F]
\nnb
&=
\frac{1}{\ell+1}\sum_{j\in [N]}{\bf 1}_{{\red \psi + {\bf 1}_j\in \Omega^\pi_{k-\ell}}}\E_{\pi(\cdot|\sigma\geq\psi)}[F{\bf 1}_{\sigma_j=1}]
\nnb
&=
\E_{\pi(\cdot|\sigma\geq\psi)}[F]
=:
\E_{\mu_{\ell+1}^\psi}[F]
,
\end{align}
where the third line comes from the fact that, 
{\red for each $\sigma\in\Omega_{N,k}$ with $\sigma\geq \psi$, 
there are $\ell+1$ sites at which to add a particle to $\psi$ that yield a configuration in $\Omega^\pi_{k-\ell}$.}

\begin{proposition}\label{prop_entropy_contraction}
Assume that $\pi$ has log-concave generating polynomial $g_\pi$. 
Then, for any $1 \leq p\leq k$,
\begin{equation}
\ent_{\pi}(F)
=\E_{\nu_k}[\ent_{\mu_k}(F)]
\leq 
\frac{k}{p}\E_{\nu_{p}}[\ent_{\mu_{p}}(F)]
.
\label{eq_accumulated_error_appendix}
\end{equation}
\end{proposition}

\begin{proof}
The first equality in \eqref{eq_accumulated_error_appendix} follows from the observation that $\mu_k=\pi$.
The inequality in \eqref{eq_accumulated_error_appendix} will be derived by reducing recursively the number of particles and using the contraction estimate~\eqref{eq_entropic_independence_appendix} which provides a bound on the loss of removing a particle.
Writing $\ent_{\ell,\ell+1}$ for the entropy associated with $P_{\ell,\ell+1}$, we are going to prove that:
\begin{equation}
\ent_{\ell,\ell+1}(\E_{\mu_{\ell}}[F])
\leq 
\frac{1}{\ell+1} \ent_{\mu_{\ell+1}}(F)
.
\label{eq_loss_removing_one_particle}
\end{equation}
Admitting~\eqref{eq_loss_removing_one_particle}, 
let us conclude the proof of \eqref{eq_accumulated_error_appendix}. 
Decomposing the measure using~\eqref{eq_decomp_pi} with $\ell=k-1$, 
noticing that $\nu_{k-1}(\cdot)=P_{k-1,k}(-1,\cdot)$ with $-1$ the configuration with no particle and using~\eqref{eq_loss_removing_one_particle},
\begin{align}
\ent_{\pi}(F)
&=
\E_{\nu_{k-1}}[\ent_{\mu_{k-1}}(F)]
+\ent_{\nu_{k-1}}\big(\E_{\mu_{k-1}}[F]\big)
\nnb
&\leq 
\E_{\nu_{k-1}}[\ent_{\mu_{k-1}}(F)]
+\frac{1}{k}
\ent_{\mu_k}(F)
.
\end{align}
As $\mu_k=\pi$, we get:
\begin{equation}
\ent_{\pi}(F)
\leq 
\frac{k}{k-1}\E_{\nu_{k-1}}[\ent_{\mu_{k-1}}(F)]
.
\label{eq: intermediaire A 19}
\end{equation}
 One can then analogously {\red use  \eqref{eq_prop_P_ell_ell+1} to decompose $\mu_{k-1} = P_{k-2,k-1} \mu_{k-2}$ and get by \eqref{eq_loss_removing_one_particle}
 \begin{equation}
\ent_{\mu_{k-1}}(F) - P_{k-2,k-1} [\ent_{\mu_{k-2}}(F)]
=  \ent_{k-2,k-1} (\E_{\mu_{k-2}}[F]) \leq 
 \frac{1}{k-1} \ent_{\mu_{k-1}}(F)
;
\end{equation}
from which we obtain by \eqref{eq: intermediaire A 19}:}
\begin{equation}
\ent_{\pi}(F)
\leq 
\frac{k}{k-1}\cdot \frac{k-1}{k-2}\E_{\nu_{k-2}}[\ent_{\mu_{k-2}}(F)]
= \frac{k}{k-2}\E_{\nu_{k-2}}[\ent_{\mu_{k-2}}(F)]
.
\label{eq_accumulated_error_appendix-pf}
\end{equation}
Iterating this procedure $p$ times proves the proposition assuming~\eqref{eq_loss_removing_one_particle}, established next. 

Fix $\psi \in{\red \Omega^\pi_{k-(\ell+1)}}$. 
By definition of $\mu^\varphi_\ell = \pi(\cdot|\sigma\geq \varphi)$ ($\varphi\in{\red \Omega^\pi_{k-\ell}}$) and $P_{\ell,\ell+1}{\red (\psi,\cdot)}$ (see~\eqref{eq_def_P_ell_ell+1}),
\begin{align}
\ent_{\ell,\ell+1}(\E_{\mu_{\ell}}[F])
=
\sum_{i\in[N]}
\E_{\mu_\ell^{\psi+{\bf 1}_i}}[F]\log \Big(\frac{\E_{\mu_\ell^{\psi+{\bf 1}_i}}[F]}{P_{\ell,\ell+1}\E_{\mu_{\ell}}[F]}\Big)P_{\ell,\ell+1}^\psi(i)
.
\label{eq_entropy_P_ell_ell+1}
\end{align} 
Recall that $P_{\ell,\ell+1}\E_{\mu_\ell}[F] = \E_{\mu_{\ell+1}}[F]$ and notice:
\begin{equation}
\E_{\mu_\ell^{\psi+{\bf 1}_i}}[F] 
=
\frac{\E_{\pi(\cdot|\sigma\geq \psi)}[{\bf 1}_{\sigma_i=1}F]}{\E_{\pi(\cdot|\sigma\geq \psi)}[{\bf 1}_{\sigma_i=1}]}
=
\frac{\E_{\mu_{\ell+1}^{\psi}}[{\bf 1}_{\sigma_i=1}F]}{\E_{\mu_{\ell+1}^{\psi}}[{\bf 1}_{\sigma_i=1}]}
.
\end{equation}
With $D_{\ell+1\to 1}$ defined in~\eqref{eq_def_D_kto1}, we find:
\begin{align}
\ent_{\ell,\ell+1}(\E_{\mu_{\ell}}[F])
&=
\frac{1}{\ell+1}\sum_{i\in[N]}{\bf 1}_{\psi_i= - 1}\E_{\mu_{\ell+1}^\psi}[F{\bf 1}_{\sigma_i=1}] \log\Bigg(\frac{\E_{\mu_{\ell+1}^{\psi,F}}[{\bf 1}_{\sigma_i=1}]}{\E_{\mu_{\ell+1}^{\psi}}[{\bf 1}_{\sigma_i=1}]}\Bigg)
\nnb
&= 
\E_{\mu^{\psi}_{\ell+1}}[F]\mathbb H\Big(\mu_{\ell+1}^{\psi,F}D_{\ell+1\to 1}\Big|\mu_{\ell+1}^{\psi}D_{\ell+1\to 1}\Big)
,
\end{align}
where $\mu_{\ell+1}^{\psi,F} = F\mu_{\ell+1}^{\psi}/\E_{\mu_{\ell+1}^\psi}[F]$. 
The measure $\mu_{\ell+1}^{\psi}$ is strongly log-concave as it is just $\pi$ conditioned on the position of some particles. 
The $1$-entropic independence~\eqref{eq_entropic_independence_appendix} therefore gives~\eqref{eq_loss_removing_one_particle}:
\begin{equation}
\ent_{\ell,\ell+1}(\E_{\mu_{\ell}}[F])
\leq 
\frac{1}{\ell+1}\E_{\mu^{\psi}_{\ell+1}}[F]\, \mathbb H\Big(\mu_{\ell+1}^{\psi,F}\Big|\mu_{\ell+1}^{\psi}\Big)
=
\frac{1}{\ell+1}\ent_{\mu^\psi_{\ell+1}}[F]
.
\qedhere
\end{equation}
\end{proof}

\subsection{Bases-exchange dynamics}
The bases-exchange dynamics is very particular in the sense that it is designed to be compatible with the decomposition in terms of the 
number of particles. 
This can be understood from the following two observations. 
First, the bases-exchange jump rates are invariant under conditioning on the position of particles. 
Indeed, let $E\subset[N]$ with $|E|\leq  k-2$, 
let $(i,j)\in ([N]\setminus E)^2$ with $i\neq j$. 
Write $\pi^E$ for the conditional measure $\pi(\cdot|\sigma_\ell=1\text{ for }\ell\in E)$ 
and let $\sigma\in\Omega_{N,k}$ satisfy $\sigma_i=1$ and $\sigma_\ell =1$ for $\ell \in E$.
Then, {\red recalling that $J_i(\sigma)$ denotes the set of sites where a particle can be added to $\sigma$ to after removing $i$ in order to obtain a configuration in the support of $\pi$:}
\begin{align}
{\red c_\pi^{\mathrm{du}}} (\sigma,\sigma^{ij}) 
&=
\frac{\pi(\sigma^{ij})}{\sum_{k\in {\red J_i(\sigma)}}\pi(\sigma^{ik})}
=
\frac{\pi^E(\sigma^{ij})}{\sum_{k\in {\red J_i(\sigma)}}\pi^E(\sigma^{ik})}
=
{\red c_{\pi^E}^{\mathrm{du}}}   (\sigma,\sigma^{ij}) 
.
\label{eq_conditioned_jp_rates}
\end{align}
This will be used to relate the jump rates for $\pi$ and $\pi$ conditional on having $+$ spins in a certain set.

The second property is the following. 
Let $\tilde P$ denote the {\red transition matrix} of the bases-exchange dynamics~\eqref{eq_mLSI_appendix}, i.e., with jump rates ${\red \tilde P}(\sigma,\sigma')=k^{-1} {\red c_\pi^{\mathrm{du}}} (\sigma,\sigma')$ (the $k^{-1}$ ensures that $\tilde P$ is a stochastic matrix). 
It has the important property that it can be decomposed as the following product of stochastic matrices:
\begin{equation}
\tilde P 
=
D_{k\to k-1}U_{k-1\to k}
,
\end{equation}
where $D_{k\to k-1}$ is the operation that removes one particle uniformly at random, and $U_{k-1\to k}$ adds a particle with probability relative to $\pi$:
\begin{align}
&D_{k\to k -1}: (\sigma,\sigma') \in\Omega_{N,k}\times{\red \Omega^\pi_{k-1}}
\longmapsto 
\frac{1}{k}{\bf 1}\Big\{\sigma' = \sigma^i\text{ for some } i\in I(\sigma)\Big\},
\nnb
&U_{k-1\to k}: (\sigma',\sigma) \in{\red\Omega^\pi_{k-1}}\times\Omega_{N,k}
\longmapsto 
{\bf 1}_{{\red \sigma\geq\sigma'}} \frac{\pi(\sigma)}{\sum_{{\red i:(\sigma' + {\bf 1}_i)\in\text{supp}(\pi)}}\pi( {\red \sigma' + {\bf 1}_i })}
.
\end{align}

\begin{lemma} \label{lem:app-two-particle}
  The modified log-Sobolev inequality of Theorem~\ref{thm:hs-lsi} holds for $k=2$:
  \begin{equation}
    \ent_{\pi}(F) \leq D^{\mathrm{du}}_{\pi}(F,\log F).
  \end{equation}
\end{lemma}

\begin{proof}
For any $F:\Omega_{N,k}\to\R_+$, writing $\pi^F$ for the probability measure $F\pi/\E_{\pi}[F]$, 
{\red we note from the reversibility of $\tilde P$ with respect to $\pi$:}
\begin{align}
\pi (\sigma) \tilde PF (\sigma)
= \E_\pi( \tilde PF \, 1_\sigma)
= \E_{\pi}[F]\, \E_{\pi^F}( \tilde P 1_\sigma)
= \E_{\pi}[F]\, \pi^F \tilde P (\sigma).
\end{align}
This implies
\begin{align}
\label{eq: entropy decomposition tilde P}
\ent_\pi(\tilde PF)
=
\E_{\pi}[F]\, \mathbb H(\pi^F\tilde P|\pi)
&\leq 
\E_{\pi}[F]\, \mathbb H(\pi^FD_{k\to k -1} |\pi D_{k\to k -1} ) , 
\end{align}
where we used Jensen inequality and the fact that $U_{k-1\to k}$ is a stochastic matrix.

For $k=2$, combining \eqref{eq: entropy decomposition tilde P}
and the contraction estimate~\eqref{eq_entropic_independence_appendix} implies 
for any $F:\Omega_{N,2}\to\R_+$
\begin{equation}
\ent_\pi(\tilde PF)
\leq 
\E_{\pi}[F] \mathbb H(\pi^FD_{2\to 1}|\pi D_{2\to 1}) 
\leq 
\frac{1}{2}\E_{\pi}[F]\mathbb H(\pi^F|\pi)
=
\frac{1}{2}\ent_\pi(F)
.
\label{eq_mLSI_2particles}
\end{equation}
{\red By standard arguments, see, e.g.,~\cite[Lemma~16]{2106.04105} or \cite{MR2283379},
this  contraction of the entropy is equivalent to the modified log-Sobolev inequality for the discrete Markov chain associated with $\tilde P$: 
\begin{equation}
\ent_{\pi}(F)
\leq 
2 \E_{\pi}[ (F  - \tilde P F)  \; \log F]
= 
D^{\mathrm{du}}_{\pi}(F,\log F),
\end{equation}
using that $ \tilde P=\frac{1}{2} {\red c_\pi^{\mathrm{du}}}$
for $k=2$ in order to recover $D^{\mathrm{du}}_{\pi}$.
This completes the proof of Lemma~\ref{lem:app-two-particle}.}
\end{proof}

\begin{proof}[Proof of Theorem~\ref{thm:hs-lsi}]
To clarify the exposition, we restrict to configurations with $k\geq 2$ particles in the following proof.
The case $k=1$ is also accessible, see~\cite{MR4203344}.  

Applying the modified log-Sobolev inequality for the two-particle measure
from Lemma~\ref{lem:app-two-particle} to $\mu_2$, we get:
\begin{equation}
\E_{\nu_{2}}[\ent_{\mu_{2}}(F)]
\leq
  \E_{\nu_{2}}[D^{\mathrm{du}}_{\mu_{2}}(F,\log F)].
  \label{eq_mLSI_2particles_applied}
\end{equation}
Using the observation~\eqref{eq_conditioned_jp_rates} that the bases-exchange jump rates for $\mu_2$ are the same as for $\pi$,
the right-hand side of~\eqref{eq_mLSI_2particles_applied} is:
\begin{equation}
  \E_{\nu_{2}}[D^{\mathrm{du}}_{\mu_{2}}(F,\log F)]
  =
  {\red \frac{1}{2}}\E_{\nu_{2}}\Big[\sum_{\sigma\in\Omega_{N,k}}\pi(\sigma|\sigma\geq \varphi)
  \sum_{i\in I(\sigma)\setminus I(\varphi)} \sum_{j\in[N]}  {\red c_\pi^{\mathrm{du}}}(\sigma,\sigma^{ij})\nabla_{ij}F(\sigma)\nabla_{ij}\log F(\sigma)\Big].
  \label{eq_DF_2_particles}
\end{equation}
Note that the sum over $j$ is written over the whole of $[N]$
as exchanges between two sites $i,j$ occupied by particles do no contribute due to $\nabla_{ij}F(\sigma) = 0$ 
and the jump rates for configurations $\sigma^{ij}\notin\text{supp}(\sigma)$ vanishes.

Under the measure $\nu_{2}\otimes \mu_{ 2}$, the variable $(\varphi,\sigma)\in{\red\Omega^{\pi}_{k- 2}}\times\Omega_{N,k}$ is distributed such that $\sigma\sim \pi$ and
$I(\varphi)= \{i: \varphi_i=1\}$ is uniform on subsets of $I(\sigma) = \{i: \sigma_i = 1\}$ of size $2$:
\begin{equation}
  \E_{\nu_{2}}\E_{\mu_{2}}[F]
  =
    \sum_{\varphi\in{\red \Omega^\pi_{k-2}}} \frac{2}{k(k-1)} \sum_{\sigma\in \Omega_{N,k}} \pi(\sigma) 1_{\sigma \geq \varphi} F(\sigma).
\label{eq: decomposition explicite}
\end{equation}
{\red Together with~\eqref{eq_DF_2_particles}, this yields:} 
\begin{align}
 &\E_{\nu_{2}}[D^{\mathrm{du}}_{\mu_{2}}(F,\log F)]  
 \nnb
 &\qquad=
{\red\frac{1}{2}}
\frac{2}{k(k-1)}\sum_{\sigma\in\Omega_{N,k}}\pi(\sigma) \sum_{i,j\in[N]}{\red c_\pi^{\mathrm{du}}  (\sigma,\sigma^{ij})} \nabla_{ij}F(\sigma)\nabla_{ij}\log F(\sigma)\sum_{\varphi\in{\red \Omega^\pi_{ k- 2}}}{\bf 1}_{\sigma\geq \varphi}{\bf 1}_{\sigma_i=1=-\varphi_i}
\nnb
&\qquad=
\frac{2}{k} D^{\mathrm{du}}_\pi(F,\log F)
,
\end{align}
where we used that, {\red for each $\sigma\in\Omega_{N,k}$}, 
there are exactly $k-1$ configurations $\varphi\in {\red \Omega^\pi_{k-2}}$ with $\sigma\geq\varphi$ and $\sigma_i=1=-\varphi_i$. 
Combining with~\eqref{eq_accumulated_error_appendix} proves the claim. 
\end{proof}

\subsection{Other dynamics}
The entropy contraction argument of Proposition~\ref{prop_entropy_contraction} is not related to the specific choice of dynamics. 
In particular, to prove a modified log-Sobolev inequality for other dynamics, it is enough to prove it  with one (or a few) particles only as we now explain.

Let $1\leq k\leq \lfloor N/2\rfloor$ and consider any family of jump rates $(q(\sigma,\sigma'))_{ij}$ on $\Omega_{N,k}$ such that $\pi$ is the stationary measure of the associated dynamics (which does not even have to be reversible). 
Proposition~\ref{prop_entropy_contraction} applies unchanged:
for any $1 \leq p\leq k$,
\begin{equation}
\ent_{\pi}(F)
=
\E_{\nu_k}\big[\ent_{\mu_k}(F)\big]
\leq 
\frac{k}{p}\E_{\nu_{p}}[\ent_{\mu_{p}}(F)]
.
\end{equation}
Take for instance $p=1$. 
In that case {\red $\mu_1 = \pi(\cdot|\sigma\geq \varphi)$ ($\varphi\in\Omega^\pi_{k-1}$)} can be identified as a measure on a single particle jumping on the {\red at most $N-k+1$ sites $i$ such that $\varphi+{\bf 1}_i\in\text{supp}(\pi)$.} 
Write $q^\varphi(i,j) = {\bf 1}_{{\red \sigma= \varphi+{\bf 1}_i}}q(\sigma,\sigma^{ij})$ for the associated jump rates.  

Let $\gamma(q,\varphi)$ be the corresponding log-Sobolev constant (the argument is identical with the modified log-Sobolev constant):
\begin{equation}
\ent_{\mu_1}(G) 
\leq 
\frac{2}{\gamma(q,\varphi)}\sum_{i\in[N]}\mu_1(i)\Big[\frac{1}{2}\sum_{j\in[N]}q^\varphi(i,j) \big[\sqrt{G(i)}-\sqrt{G(j)}\big]^2\Big]
,\qquad
G:[N]\to\R_+
.
\end{equation}
Then:
\begin{align}
\ent_{\pi}(F)
&\leq 
k\E_{\nu_{1}}\Big[\frac{2}{\gamma(q,\varphi)}{\red \frac{1}{2}}\sum_{i,j\in[N]} \mu_1(i) q^\varphi(i,j) \big[\sqrt{F}(\varphi+{\bf 1}_j) - \sqrt{F}(\varphi+{\bf 1}_{i})\big]^2\Big]
\nnb
&= 
\sum_{\sigma\in\Omega_{N,k}}\pi(\sigma) \sum_{i,j\in[N]}q(\sigma,\sigma^{ij}) [\nabla_{ij}\sqrt{F}(\sigma)]^2 {\red \Big(\sum_{\varphi\in\Omega^\pi_{k-1}}\frac{1}{\gamma(q,\varphi)}{\bf 1}_{\varphi=\sigma-{\bf 1}_i}\Big)
}
\nnb
&=
\sum_{\sigma\in\Omega_{N,k}}\pi(\sigma) \sum_{i,j\in[N]}\frac{1}{\gamma(q,\sigma-{\bf 1}_i)} q(\sigma,\sigma^{ij}) [\nabla_{ij}\sqrt{F}(\sigma)]^2
.
\end{align}
Proving the log-Sobolev inequality for the $k$-particle dynamics is therefore reduced to a bound on the $1$-particle log-Sobolev constant uniformly on subsets of $[N]$ with size $N-(k-1)$.

\section*{Acknowledgements}

We thank the referees and Jakob Kellermann for helpful suggestions that improved the presentation.

This work was supported by the European Research Council under the European Union's Horizon 2020 research and innovation programme
(grant agreement No.~851682 SPINRG). 

\bibliography{all}
\bibliographystyle{plain}

\end{document}